\documentclass[leqno,10pt]{amsart}

\usepackage{amsmath,amssymb,mathrsfs} 
\usepackage[dvipdfmx, 
colorlinks,linkcolor=blue,citecolor=blue,urlcolor=red]{hyperref}
\usepackage{times}
\usepackage[all]{xy}

\theoremstyle{plain}
    \newtheorem{thm}{Theorem}[section]
    \newtheorem{ppn}[thm]{Proposition}
    \newtheorem{lem}[thm]{Lemma}
    \newtheorem{cor}[thm]{Corollary}
\theoremstyle{definition}

\theoremstyle{remark}
    \newtheorem{rmk}[thm]{Remark}
    \newtheorem*{rmk*}{Remark}   
    \newtheorem{epl}[thm]{Example}

\numberwithin{equation}{section}

\allowdisplaybreaks

\def\Ker{\operatorname{Ker}}\def\Hom{\operatorname{Hom}}\def\Spec{\operatorname{Spec}}
\def\C{\mathbb{C}}\def\Q{\mathbb{Q}}\def\Z{\mathbb{Z}}\def\F{\mathbb{F}}
\def\ol#1{\overline{#1}}\def\ul#1{\underline{#1}}\def\wt#1{\widetilde{#1}}\def\wh#1{\widehat{#1}}
\def\os#1#2{\overset{#1}{#2}}
\def\ot{\otimes}
\def\a{\alpha}\def\b{\beta}\def\d{\delta}
\def\io{\iota}\def\k{\kappa}\def\z{\zeta}\def\x{\xi}\def\y{\eta}

\def\G{\Gamma}

\def\L{\Lambda}\def\s{\sigma}\def\t{\tau}
\def\Im{\operatorname{Im}}
\def\prim{\mathrm{prim}}

\def\rank{\operatorname{rank}}

\newcommand{\CH}{\operatorname{CH}}

\def\D{\Delta}

\def\r{\rho}

\newcommand\sC{\mathscr{C}}

\newcommand\sO{\mathscr{O}}

\newcommand\ba{\mathbf{a}}

\newcommand{\End}{\mathrm{End}}
\newcommand{\Gal}{\mathrm{Gal}}

\newcommand{\A}{{\mathbb{A}}}

\newcommand{\lra}{\longrightarrow}

\newcommand{\sFr}{\mathfrak{Fr}}

\newcommand{\bchi}{{\boldsymbol\chi}}
\newcommand{\bn}{{\boldsymbol\nu}}

\renewcommand{\P}{\mathbb{P}}

\newcommand\n{\nu}\newcommand\m{\mu}
\newcommand\vp{\varphi}

\newcommand\ck{\wh{\k^*}}
\newcommand\0{\circ}

\newcommand{\Fr}{\mathrm{Fr}}
\newcommand{\pr}{\mathrm{pr}}
\newcommand{\Tr}{\mathrm{Tr}}

\newcommand{\Cor}{\mathbf{Cor}}

\newcommand{\Sm}{\mathbf{Sm}}
\newcommand{\SmProj}{\mathbf{SmProj}}
\newcommand{\DM}{\mathbf{DM}_{\mathrm{gm}}}
\newcommand{\Chow}{\mathbf{Chow}}
\newcommand{\id}{\mathrm{id}}

\newcommand{\Pic}{\mathbf{Pic}}

\begin{document}

\title{Motivic Gauss and Jacobi sums}
\author{Noriyuki Otsubo and Takao Yamazaki}
\address{Department of Mathematics and Informatics, Chiba University, Inage, Chiba, 263-8522 Japan}
\email{otsubo@math.s.chiba-u.ac.jp}
\address{Department of Mathematics,
Chuo University, 1-13-27 Kasuga,
Bunkyo-ku, Tokyo 112-8551, Japan}
\email{ytakao@math.chuo-u.ac.jp}
\begin{abstract}
We study the Gauss and Jacobi sums from a viewpoint of motives.
We exhibit isomorphisms between Chow motives
arising from the Artin-Schreier curve and the Fermat varieties over a finite field,
that can be regarded as (and yield a new proof of)
classically known relations among Gauss and Jacobi sums
such as Davenport-Hasse's multiplication formula.
As a key step,
we define motivic analogues of the Gauss and Jacobi sums as 
algebraic correspondences,
and show that they represent the Frobenius endomorphisms of such motives.
This generalizes Coleman's result for curves.
These results are applied to 
investigate the group of invertible Chow motives with coefficients in a cyclotomic field.
\end{abstract}

\date{\today.}
\subjclass[2010]{14C15 (Primary) 11L05, 19E15 (Secondary)}
\keywords{Gauss sums, Jacobi sums, Davenport-Hasse relation, Chow motives, Weil numbers}

\thanks{
The first author is supported by JSPS KAKENHI Grant (JP22H00096). 
The second author is supported by JSPS KAKENHI Grant (JP21K03153). 
}

\maketitle

%%%%%%%%%%%%%%%%%%%%%%%%%

%%%
\section{Introduction}

Let $\k$ be a finite field of cardinality $q$ and of characteristic $p$,
and $d$ a positive divisor of $q-1$.
Take a non-trivial additive character $\psi \colon \k \to \C^*$,
and multiplicative characters $\chi, \chi_1, \dots, \chi_n \colon \mu_d \to \C^*$,
where $\mu_d := \{ m \in \k^* \mid m^d=1 \}$.
We consider the Gauss sum $g(\psi,\chi) \in \Q(\zeta_{pd})$ and
the Jacobi sum $j(\chi_1,\dots, \chi_n) \in\Q(\zeta_d)$
(see \eqref{eq:def-Gsum}, \eqref{eq:def-Jsum} for the definitions),
where $\zeta_k := e^{2 \pi i/k} \in \C$.
In this introduction,
we discuss the following relations among $g(\psi, \chi)$ and
$j(\chi_1, \dots, \chi_n)$:

\begin{itemize}
\item 
Assume that none of $\chi_1, \dots, \chi_n, \prod_{i=1}^n \chi_i$ is trivial.
Then we have (cf. \eqref{e1})
\begin{align}
\label{e0-1}
&\prod_{i=1}^ng(\psi,\chi_i)
=
g(\psi, \chi_1\cdots \chi_n)j(\chi_1,\dots, \chi_n).
\end{align}
\item 
Assume $n \mid d$ and
let $\alpha \colon \m_d \to \C^*$ be a character 
such that $\alpha^n \not= 1$.
Then we have the Davenport-Hasse multiplication formula
(cf. \eqref{mult-j})
\begin{equation}\label{e0-3}
\a^n(n) j(\underbrace{\a,\dots, \a}_\text{$n$ times})=\prod_{\chi^n=1,\chi\ne 1} j(\a,\chi),
\end{equation}
where $\chi$ ranges over all non-trivial characters of $\m_d$
such that $\chi^n=1$.
\end{itemize}
(See \eqref{gauss-reflection}, \eqref{gauss-bar}, \eqref{e1-reflextion}, \eqref{eq:j-induction}, 
\eqref{g-dh}, \eqref{j-dh}, \eqref{eq:mult-g}, \eqref{mult-j}
for other relations considered in the body of the text.)
The aim of the present note is to upgrade these relations to \emph{motives}.
This is archived in two steps.
The first step is to construct isomorphisms between suitable motives,
and the second is to relate the Frobenius endomorphisms
with the motivic Gauss and Jacobi sums.

To state our results, we introduce more notations.
Let $\Chow(\k, \Lambda)$ be the category of Chow motives over $\k$ 
with coefficients in a field $\Lambda$ of characteristic zero.
Let $A_d$ be the (smooth projective) Artin-Schreier curve defined by $x^{q}-x=y^d$.
We construct an object $h(A_d)^{(\psi, \chi)}$ 
of $\Chow(\k, \Q(\zeta_{pd}))$
as a direct factor of the motive $h(A_d)$ of $A_d$
cut out by the action of $\k \times \mu_d$.
Similarly, for each $c \in \k^*$
we construct an object
$h(F_d^{(n)}\langle c \rangle)^{(\chi_1, \dots, \chi_n)}$ 
of $\Chow(\k, \Q(\zeta_{d}))$
as a direct factor of the motive $h(F_d^{(n)}\langle c \rangle)$
of 
the Fermat variety $F_d^{(n)}\langle c \rangle \subset \P^n$ defined by 
$u_1^d+ \cdots + u_n^d=cu_0^d$ cut out by the action of $\mu_d^n$.
(See section \ref{sect:AS-F} for details.)
We drop $\langle c \rangle$ from the notation when $c=1$.
Our first main result is the following.

\begin{thm}\label{thm1}\ 
\begin{enumerate}
\item 
Let $\psi$ be a character of $\k$ and $\chi_1, \dots, \chi_n$ characters of $\m_d$.
If none of $\psi, \chi_1,\dots, \chi_n$, $\prod_{i=1}^n \chi_n$ is trivial, 
then there exists an isomorphism between invertible objects 
of $\Chow(\k,\Q(\zeta_{pd}))$
\begin{align*}
&\bigotimes_{i=1}^n h(A_d)^{(\psi, \chi_i)} 
\simeq 
h(A_d)^{(\psi, \prod_{i=1}^n \chi_i)} \ot h(F_d^{(n)})^{(\chi_1,\dots, \chi_n)}.
\end{align*}
\item
Suppose that $n$ divides $d$ and 
let $\alpha$ be a character of $\m_d$ such that $\alpha^n \not= 1$.
Then there exists an isomorphism 
between invertible objects of $\Chow(\k,\Q(\zeta_{d}))$
\begin{align*}
h(F_{d}^{(n)}\langle n\rangle)^{(\a,\dots,\a)} \simeq 
\bigotimes_{\chi^n=1,\chi\ne 1} h(F_{d}^{(2)})^{(\a,\chi)}. 
\end{align*}
\end{enumerate}
\end{thm}

In fact, the isomorphism in (ii) holds over arbitrary base field $\k$
as long as $\k$ contains a primitive $d$th root of unity (see subsection \ref{ss7.2}). 
When $\k=\C$, it has an implication on the gamma function (see Remark \ref{rem-gamma-mult}). 

To state the second main result,
we introduce an element of the group ring $\Z[\m_d^n]$ 
\[j_d^{(n)}\langle c \rangle = (-1)^{n-1} \sum_{m_1, \dots, m_n \in \k^*, \sum m_i =c} (m_1^\frac{q-1}{d},\dots, m_n^\frac{q-1}{d}),
\] 
which we call the \emph{twisted Jacobi sum element}.
It follows from the definition that
$j_d^{(n)}\langle c\rangle$ acts on 
$h(F_d^{(n)}\langle c \rangle)^{(\chi_1, \dots, \chi_n)}$
as the multiplication by $\chi_1 \cdots \chi_n(c^\frac{q-1}{d}) j(\chi_1, \dots, \chi_n)$.
The following theorem generalizes Coleman's result \cite[Theorem A]{coleman}
for $n=2$ and $c=1$. 

\begin{thm}\label{thm2}
Let $\chi_1, \dots, \chi_n$ be characters of $\m_d$
such that none of $\chi_1, \dots, \chi_n, \prod_{i=1}^n \chi_i$ is trivial.
Then 
the endomorphism of $h(F_d^{(n)}\langle c \rangle)^{(\chi_1, \dots, \chi_n)}$ 
induced by $j_d^{(n)}\langle c\rangle$
agrees with the Frobenius endomorphism.
\end{thm}

Coleman also proved that the \emph{Gauss sum element}
\begin{equation}\label{eq:coleman}
g_d =-\sum_{m \in \k^*} (m,m^\frac{q-1}{d}) \in \Z[\k \times \mu_{d}] 
\end{equation}
induces the Frobenius endomorphism on $h(A_d)^{(\psi, \chi)}$ 
if $\psi$ and $\chi$ are non-trivial.
Since the Frobenius endomorphism commutes with any morphisms
(see \eqref{eq:Fr-comm-any-f} below),
%by \cite[p. 80]{Kleiman2} (cf. \cite[Proposition 2]{Soule}),
\eqref{e0-1} and \eqref{e0-3}
can be deduced from Coleman's result
and Theorems \ref{thm1}, \ref{thm2}.
See Remark \ref{rem:deduce-formulas} for details.

We conclude this introduction by a discussion on 
the relations among Weil numbers and motives.
Recall that $\alpha \in \C$ is called a \emph{$q$-Weil number} of weight $w \in \Z$
if there exists $m \in \Z$ such that
$q^m \alpha$ is an algebraic integer
and 
$|\sigma(\alpha)|=q^{w/2}$ for all $\sigma \colon \Q(\alpha) \to \C$.
Let $W_q(\L)$ be the subgroup of $\L^*$
consisting of all $q$-Weil numbers (of arbitrary weight) belonging to 
a given subfield $\Lambda$ of $\C$.
It is conjectured by Beilinson \cite[1.0]{B}
that the rational equivalence and numerical equivalence
should agree over a finite field $\k$ (with coefficients in $\L$).
If we assume this as well as the Tate conjecture, 
it follows from \cite[Proposition 2.21]{Milne} that
all simple objects of $\Chow(\k, \ol{\Q})$ should be invertible,
and their isomorphism classes would form a group 
(with respect to the tensor product) isomorphic to $W_q(\ol{\Q})$.
In particular,
the group $\Pic(\Chow(\k, \L))$ of all isomorphism classes of invertible objects  of $\Chow(\k, \L)$
should be isomorphic to a subgroup of $W_q(\L)$
by Lemma \ref{lem:scalar-ext} below.
Therefore, the multiplicative relations among $q$-Weil numbers 
(such as \eqref{e0-1} and \eqref{e0-3})
\emph{should} come from relations among motives,
as demonstrated by Theorem \ref{thm1}.

Using our motivic relations,
we shall deduce the following two results on $\Pic(\Chow(\k, \Q(\zeta_d)))$ where $\k$ is the residue field of $\Q(\z_d)$ at a prime $v\nmid d$, 
both conditional to the conjectures of Beilinson and Tate
(see Corollary \ref{cor:conditional} and the discussion after Proposition \ref{prop:basis}):
\begin{itemize}
\item 
$\Pic(\Chow(\k, \Q(\zeta_d))$ should be 
generated by the Fermat motives $h(F_d^{(2)})^{(\chi_1, \chi_2)}$,
up to powers and Artin motives.
\item 
All the relations among $h(F_d^{(2)})^{(\chi_1, \chi_2)}$ in $\Pic(\Chow(\k, \Q(\zeta_d)))$ 
should be implied by Theorem \ref{thm1} (ii)
and the reflection relation \eqref{j-reflextion-mot},
up to powers and Artin motives.
\end{itemize}
The key input here is a result of Iwasawa-Sinnott \cite{sinnott} 
on Stickelberger's ideal.

The paper is organized as follows.
After a brief recollection on the Gauss and Jacobi sums,
we define their motivic variants in section 2.
We prepare a few basic facts on the Chow and Voevodsky motives in section 3.
We then extensively study the motives of the Artin-Schreier curves and the Fermat varieties in section 4,
where a crucial ingredient is 
the inductive structure of Fermat varieties due to Katsura-Shioda \cite{shioda-katsura}.
We complete the proof of Theorem \ref{thm1} (i), (ii) and Theorem \ref{thm2}
in sections 5, 7 and 6, respectively.
The last section 8 is devoted to a discussion on Weil numbers
and $\Pic(\Chow(\k, \Q(\zeta_d)))$.

\subsection*{Acknowledgements}
The authors thank sincerely the referee for his/her careful reading and helpful comments.

%%%
\section{Gauss and Jacobi sums}

Let $\k$ be a finite field of characteristic $p$ with $q$ elements,
and let $d$ be a positive divisor of $q-1$.
We write $\wh G:=\Hom(G,\C^\times)$ when $G$ is a finite group,
and $\ol{\chi}:=\chi^{-1}$ for $\chi \in \wh G$.
For $\psi \in \wh\k \setminus \{ 1\}$ and  $\chi \in \wh\m_d$, 
the {\em Gauss sum} is defined by 
\begin{equation}\label{eq:def-Gsum}
g(\psi,\chi)=-\sum_{m\in\k^*} \psi(m)\chi(m^{\frac{q-1}{d}}) \ \in \Q(\zeta_{pd}).
\end{equation}
Note that $g(\psi, 1)=1$. 
We have for any $\chi \not= 1$
\begin{equation}\label{gauss-reflection}
g(\psi,\chi)g(\ol\psi,\ol\chi)=q.
\end{equation} 
In particular, $g(\psi,\chi)$ is a $q$-Weil number of weight one if $\chi \ne 1$. 
We have also
\begin{equation}\label{gauss-bar}
g(\ol\psi,\chi)=\chi((-1)^{\frac{q-1}{d}})g(\psi,\chi). 
\end{equation}
If $d' \mid d$ and if $\chi' \in \wh\m_{d'}$ is such that $\chi(m)=\chi'(m^{d/d'})$
for all $m \in \m_d$,
then we have 
$g(\psi, \chi) = g(\psi, \chi')$. 

For $\chi_1,\dots, \chi_n \in \wh\m_d$, the {\em Jacobi sum} is defined by 
\begin{equation}\label{eq:def-Jsum}
j(\chi_1,\dots, \chi_n)=(-1)^{n-1} \sum_{m_i\in\k^*, \sum_{i=1}^nm_i=1} \chi_1(m_1^{\frac{q-1}{d}})\cdots
\chi_n(m_n^{\frac{q-1}{d}}) \ \in\Q(\zeta_d).
\end{equation}
If none of $\chi_1, \dots, \chi_n, \prod_{i=1}^n \chi_i$ is trivial,
we have (cf. \cite[Proposition 2.2]{otsubo}) 
\begin{align}
\label{e1}
&j(\chi_1,\dots, \chi_n) =\frac{g(\psi,\chi_1)\cdots g(\psi,\chi_n)}{g(\psi, \chi_1\cdots \chi_n)},
\\
\label{e1-reflextion}
&j(\chi_1,\dots, \chi_n)j(\ol{\chi_1},\dots, \ol{\chi_n}) = q^{n-1}.
\end{align}
(In particular, the right member of \eqref{e1} is independent of $\psi$).
If $d' \mid d$ and if $\chi_i' \in \wh\m_{d'}$ is such that $\chi_i(m)=\chi_i'(m^{d/d'})$
for all $m \in \m_d$ and $i=1, \dots, n$,
then we have $j(\chi_1,\dots, \chi_n) = j(\chi_1',\dots, \chi_n')$.

\begin{rmk}\label{gamma-beta}
The relations \eqref{gauss-reflection} and \eqref{e1} are finite field analogues of the functional equations for the gamma and  beta functions: 
\[\G(s)\G(1-s)=\frac{\pi}{\sin \pi s}, \quad B(s_1,\dots, s_n)=\frac{\G(s_1)\cdots \G(s_n)}{\G(s_1+\cdots+s_n)}.\]
\end{rmk}

Define the {\em Gauss sum element} in the group ring $\Z[\k\times \m_d]$ by 
\[g_d =-\sum_{m \in \k^*} (m,m^\frac{q-1}{d}).\]
Define for $c \in \k$ the {\em twisted Jacobi sum element} in the group ring $\Z[\m_d^n]$ by 
\begin{equation}\label{eq:tw-js}
j_d^{(n)}\langle c \rangle =(-1)^{n-1} \sum_{m_1,\dots, m_n \in \k^*, \sum_{i=1}^n m_i=c}
(m_1^\frac{q-1}{d}, \dots, m_n^\frac{q-1}{d}).
\end{equation}
We write $j_d^{(n)}=j_d^{(n)}\langle 1 \rangle$. 
Note that if $c \not= 0$
\[j_d^{(n)}\langle c \rangle = (c^\frac{q-1}{d},\dots, c^\frac{q-1}{d}) \cdot j_d^{(n)}.\]
If $d' \mid d$, then the map $\Z[\k\times \m_d] \to \Z[\k\times \m_{d'}]; (a,m)\mapsto (a,m^{d/d'})$ sends $g_d$ to $g_{d'}$, and the map $\Z[\m_d^n] \to \Z[\m_{d'}^n]; (m_i) \mapsto (m_i^{d/d'})$ sends $j_d^{(n)}\langle c \rangle$ to $j_{d'}^{(n)}\langle c \rangle$. 

If $G$ is a finite group and $\chi \in \wh G$, 
we write
\begin{equation}\label{eq:def-e-chi}
e^\chi =e_G^\chi=\frac{1}{|G|} \sum_{g \in G} \ol\chi(g)g \ \in \C[G]
\end{equation}
for the corresponding projector.
We have 
$g\cdot e^\chi = \chi(g) e^\chi$ and 
$e^\chi e^{\chi'}=\d(\chi,\chi')e^\chi$ in $\C[G]$
for any $g\in G$ and $\chi, \chi' \in \wh G$,
where $\d$ is the Kronecker delta. 
If $G$ is abelian,
we also have $\sum_{\chi\in\wh G} e^\chi=1$.
The following lemma is an immediate consequence of the definitions
and will be used frequently without further notice.

\begin{lem}\ 
\begin{enumerate}
\item For any $\psi\in\wh\k\setminus\{1\}$ and $\chi\in\wh\m_d$, we have 
\[g_d \cdot e^{(\psi,\chi)} = g(\psi,\chi) e^{(\psi,\chi)}
\ \text{ in } \C[\k\times \m_d].
\]
\item For any $\chi_1,\dots, \chi_n \in \wh\m_d$, we have 
\[j_d^{(n)}\langle c \rangle \cdot e^{(\chi_1,\dots, \chi_n)} =\chi_1\cdots\chi_n(c^\frac{q-1}{d}) j(\chi_1,\dots, \chi_n) e^{(\chi_1,\dots, \chi_n)}
\ \text{ in } \C[\m_d^n].
\]
\end{enumerate}
\end{lem}

%%%
\section{Preliminaries on motives}

In this section
$\k$ is an arbitrary field,
which we assume to be perfect from subsection \ref{sect:DM} onward.
Let $\SmProj(\k)$ be
the category of smooth projective varieties over $\kappa$.
We also fix a field $\Lambda$ of characteristic zero.
We write $A_\Lambda := A \otimes_\Z \Lambda$
when $A$ is an abelian group.

\subsection{Chow motives}

Let $\Chow(\kappa, \Lambda)$ be
the \emph{homological} category of Chow motives
over $\kappa$ with coefficients in $\Lambda$
(cf. e.g. \cite[Chapter 20]{MVW}; this is opposite of the one in \cite{Scholl}).
This is a $\Lambda$-linear rigid tensor pseudo-abelian category.
Recall that an object of $\Chow(\k, \L)$ 
can be written as a triple $(X, \pi, r)$
where $X \in \SmProj(\k)$ is  equi-dimensional,
$\pi \in \CH_{\dim X}(X \times X)_\Lambda$ is such that $\pi^2 = \pi$
(with respect to the composition of algebraic correspondences),
and $r \in \Z$.
For two such triples $(X, \pi, r)$ and $(Y, \rho, s)$,
we have
\[ \Hom_{\Chow(\k, \L)}((X, \pi, r), (Y, \rho, s)) 
= \rho \circ \CH_{\dim X + r - s}(X \times Y)_\Lambda \circ \pi.
\]
For $r=0$ we abbreviate $(X, \pi, 0)=(X, \pi)$.

The tensor product on $\Chow(\k, \Lambda)$ is given by 
\[(X, \pi, r) \otimes (Y, \rho, s)=(X\times Y, \pi \times \rho, r+s).\]
We put
$\Lambda(r):=(\Spec \k, \id_{\Spec \k}, r)$
and 
$\Lambda := \Lambda(0)$.
For any $M \in \Chow(\k, \L)$,
we set $M(r) := M \otimes \Lambda(r)$
and write $M^\vee$ for the (strong) dual of $M$.
We have $(X, \pi, r)^\vee(\dim X)=(X, {}^t\pi, -r)$,
where ${}^t \pi$ denotes the transpose of $\pi$.

Suppose $X \in \SmProj(\k)$ is connected of dimension $m$.
Given a $\k$-rational point $x_0 \in X(\k)$,
we define objects in $\Chow(\k,\L)$ by
\begin{align}
\label{eq:def-h0}
h_0(X):=(X, [X \times x_0]) \simeq \Lambda,
\quad
h_{2m}(X):=(X, [x_0 \times X]) \simeq \Lambda(m).
\end{align}
If $m=1$, we further put
\begin{align}
\label{eq:def-h1}
h_1(X):=(X, \id_X-[X \times x_0]-[x_0 \times X])
\in \Chow(\k,\L).
\end{align}
We do not indicate $x_0$ to ease the notation,
although these objects depend on the class of $x_0$ in $\CH_0(X)_\Lambda$.  

There is a \emph{covariant} functor
\[h \colon \SmProj(\k) \to \Chow(\k, \Lambda),  \quad h(X)=(X, \id_X).\]
For a morphism $f\colon X \to Y$ in $\SmProj$, we have
\begin{align*}
&f_*:=h(f)=[\Gamma_f] \colon h(X) \to h(Y),
\end{align*}
where $\Gamma_f \subset X \times Y$ denotes the graph of $f$.
If  $X, Y$ are equi-dimensional, we also have
\begin{align*}
&f^*:=[{}^t\Gamma_f] \colon h(Y) \to h(X)(\dim Y - \dim X).
\end{align*}

\begin{lem}
Let $f\colon X \to Y$ be a generically finite morphism of degree $d$,
where $X, Y \in \SmProj(\k)$ are both irreducible and of the same dimension $m$.
We define a group homomorphism
\begin{equation}\label{eq:f-sharp}
f_\# \colon \End(h(X)) \to \End(h(Y)),
\qquad
f_\#(\a) = \frac{1}{d} (f_*\circ \a \circ f^*),
\end{equation}
where $\End$ denotes the endomorphism ring in $\Chow(\k, \L)$.
\begin{enumerate}
\item 
We have a commutative diagram 
\[\xymatrix{
\End(h(X)) \ar[rr]^{f_\#} \ar@{=}[d] & & \End(h(Y)) \ar@{=}[d]\\
\CH_{m}(X\times X)_\L \ar[rr]^{(1/d)(f\times f)_*} & & \CH_{m}(Y\times Y)_\L.
}\]
\item 
Let $\s_X$ (resp. $\s_Y$) 
be an automorphism of $X$ (resp. $Y$) such that $f\circ \s_X=\s_Y\circ f$. 
Then, $f_\#((\s_X)_*)=(\s_Y)_*$.
\end{enumerate}
\end{lem}
\begin{proof}
(i) is a consequence of 
Lieberman's lemma (cf. e.g. \cite[Lemma 2.1.2]{murre}).
To see (ii), we consider the commutative diagram
\[
\xymatrix{
X \times X \ar[rr]^{f \times f} & &
Y \times Y 
\\
X \ar[u]^{\id_X \times \sigma_X} \ar[rr]_f & &
Y \ar[u]_{\id_Y \times \sigma_Y}.
}
\]
We then compute
\begin{align*}
(f \times f)_* ([\Gamma_{\s_X}])
&= ((f \times f) \circ (\id_X \times \s_X))_*([X])
\\
&= ((\id_Y \times \s_Y) \circ f)_*([X])
= (\id_Y \times \s_Y)_*(d[Y])
= d[\Gamma_{\s_Y}].
\end{align*}
We now apply (i) to conclude (ii).
\end{proof}

\begin{epl}
We will use this lemma in the following situations. 
\begin{enumerate}
\item If $f$ is an isomorphism, then $f^*=(f_*)^{-1}$ and $f_\#$ is a ring isomorphism.  
\item
If finite groups $G$ and $G'$ act on $X$ and $Y$ respectively,  and the actions are compatible under $f$ and a homomorphism $g\colon G \to G'$, then the diagram 
\[\xymatrix{
\L[G] \ar[r]^g \ar[d]& \L[G'] \ar[d]\\
\End(h(X)) \ar[r]^{f_\#} &\End(h(Y))
}\]
commutes. Hence the restriction of $f_\#$ to the image of $\L[G]$ is a ring homomorphism. 
\item
Let $\k$ be a finite field with $q$ elements and $\Fr_X$ be the $q$th power Frobenius endomorphism of $X$. 
Then $f_\#(\Fr_X)=\Fr_Y$ holds, because we have
\begin{equation}\label{eq:Fr-comm-any-f}
f \circ \Fr_X=\Fr_Y\circ f
\end{equation}
for any morphism $f \colon X \to Y$ in $\Chow(\k, \Lambda)$
by \cite[p. 80]{Kleiman2}
(see also \cite[Proposition 2]{Soule}).
\end{enumerate}\end{epl}

For later use, we state an elementary lemma.

\begin{lem}\label{lem:scalar-ext}
Let $\Lambda'$ be a field extension of $\Lambda$.
Then the scalar extension functor $\Chow(\k, \L) \to \Chow(\k, \L')$
is conservative.
\end{lem}
\begin{proof}
Let $f \colon M \to N$ be a morphism in $\Chow(\k, \Lambda)$
and assume that 
its scalar extension 
$f_{\Lambda'} \colon M_{\Lambda'} \to N_{\Lambda'}$ 
is an isomorphism in $\Chow(\k, \Lambda')$.
We must show that $f$ is an isomorphism in $\Chow(\k, \Lambda)$.
By Yoneda's lemma, it suffices to show that 
\[ f_* \colon \Hom_{\Chow(\k, \L)}(L, M) \to \Hom_{\Chow(\k, \L)}(L, N) \] 
is bijective for any $L \in \Chow(\k, \L)$.
This follows from the bijectivity of 
\[ f_{\Lambda' *} \colon \Hom_{\Chow(\k, \L')}(L_{\Lambda'}, M_{\Lambda'}) 
\to \Hom_{\Chow(\k, \L')}(L_{\Lambda'}, N_{\Lambda'}) 
\] 
since $\L'$ is faithfully flat over $\L$.
\end{proof}

\subsection{Voevodsky motives} \label{sect:DM}
From now on we assume that $\k$ is perfect.
Let $\DM(\k, \L)$ be Voevodsky's category
of geometric mixed motives over $\k$ with coefficients in $\L$
(cf., e.g. \cite{MVW}).
This is a $\Lambda$-linear rigid tensor pseudo-abelian triangulated category 
equipped with a covariant functor
$M \colon \Sm(\k) \to \DM(\k, \L)$,
where $\Sm(\k)$ is the category of 
smooth separated schemes of finite type over $\k$.
There is a fully faithful tensor functor
\begin{equation}\label{eq:Chow-DM}
\ol{M} \colon \Chow(\k, \L) \to \DM(\k, \L)
\end{equation}
such that $M \circ i =\ol{M} \circ h$,
where $i \colon \SmProj(\k) \to \Sm(\k)$ is the inclusion functor.
This fact is first proved 
by Voevodsky \cite[Corollary 4.2.6]{VoTmot} 
(see also \cite[Proposition 20.1]{MVW})
under the assumption that $\k$ admits resolution of singularities.
See \cite[6.7.3]{DG} for a proof over an arbitrary perfect base field.
For each $i \in \Z$ and an object $N$ of $\DM(\k, \L)$, 
we put $\Lambda(i) := \ol{M}(\Lambda(i))[-2i]$
and $N(i) := N \otimes \Lambda(i)$,
where $[-]$ denotes the shift functor.

For later use, we record the blow-up formula:

\begin{ppn}\label{prop:blow-up}
Let $V$ be a smooth variety over $\k$
and $Z \subset V$ a smooth closed subvariety of pure codimension $c$.
Let $f \colon U \to V$ be the blow-up along $Z$.
Then we have an isomorphism in $\DM(\k, \Lambda)$
\[
M(U) \simeq M(V) \oplus \left( \bigoplus_{i=1}^{c-1} M(Z)(i)[2i] \right).
\]
If further $V$ is projective over $\k$,
then $f$ also induces an isomorphism
$h(U) \simeq h(V) \oplus \left( \bigoplus_{i=1}^{c-1} h(Z)(i) \right)$
in $\Chow(\k, \L)$.
\end{ppn}
\begin{proof}
The first statement is 
the blow-up formula \cite[Corollary 15.13]{MVW},
and the second follows from \eqref{eq:Chow-DM}.
(The latter is also seen from \cite[Theorems 2.5, 2.8]{Scholl}.)
\end{proof}

If a finite group $G$ acts (from left) on a motive $M$ (that is, an object of either
$\Chow(\k, \L)$ or $\DM(\k, \L)$),
we write $M^\chi$ for the image of the projector
$e^\chi$ from \eqref{eq:def-e-chi}
for  $\chi \in \wh{G}$.

\begin{ppn}\label{localization}
Let $G$ be a finite group and $\chi \in \wh G$. 
Suppose that $\L$ is large enough to contain the values of $\chi$. 
Let $U, V$ be smooth varieties with $G$-action,
and let $f\colon U \to V$  be a $G$-equivariant morphism.
If one of the following conditions is satisfied,
then $f$ induces an isomorphism $M(U)^\chi \simeq M(V)^\chi$ in $\DM(k,\L)$.

\begin{enumerate}
\item 
$f\colon U \to V$ is an open immersion such that
each irreducible component $T$ of $V \setminus f(U)$ 
is smooth and $G$-stable with $M(T)^\chi=0$.
\item 
$f \colon U \to V$ is a finite generically Galois morphism
with $\Gal(U/V) \subset G$ such that $\chi|_{\Gal(U/V)}=1$.
\end{enumerate}
If further $U, V$ are projective over $\k$,
then $f$ also induces an isomorphism
$h(U)^\chi\simeq h(V)^\chi$ in $\Chow(\k, \L)$.
\end{ppn}

\begin{proof}
The first statement 
for the case (i) and (ii)
follows respectively from 
the localization sequence \cite[Theorem 15.15]{MVW}
and Lemma \ref{lem:fin-cor} below.
The second follows from \eqref{eq:Chow-DM}.
\end{proof}

In the following lemma,
we denote by $\Cor(X, Y)$ the
group of finite correspondences for $X, Y \in \Sm(\k)$
(cf. \cite[Definition 1.1]{MVW}).

\begin{lem}\label{lem:fin-cor}
If $f\colon U\to V$ is a finite generically Galois morphism in $\Sm(\k)$,
then we have equalities
\[
f_* \circ f^* = (\deg f) \cdot \id_V \quad \text{in} \  \Cor(V, V),
\quad
f^* \circ f_* = \sum_{g \in \Gal(U/V)} g_* \quad \text{in} \  \Cor(U, U).
\]
\end{lem}
\begin{proof}
For an open dense immersion $j \colon V' \to V$,
we have injections
\[
\xymatrix{
\Cor(V', V')  \ar@{^{(}->}[r]_{j \circ -} \ 
&
\Cor(V', V)
&
\ \Cor(V, V)  \ar@{_{(}->}[l]^{- \circ j}. 
}
\]
Putting $U':=f^{-1}(V')$, we have
\[
(f_* \circ f^*) \circ j = j \circ (f|_{U'*} \circ f|_{U'}^*),
\quad
\id_V \circ j = j \circ \id_{V'}
\quad
\ \text{in} \ \Cor(V', V),
\]
which reduces the first statement 
to the case $V$ is the spectrum of a field.
A similar argument reduces the second statement to the same case.
Then both statements are found in \cite[Exercise 1.11]{MVW}.
\end{proof}

\subsection{Invertible objects}

Let $\sC$ be a $\L$-linear rigid tensor pseudo-abelian category
such that the endomorphism ring of the unit object
is canonically isomorphic to $\Lambda$.
(We shall apply the following discussion
to $\sC=\Chow(\k, \L), \DM(\k, \L)$.)
Recall that 
an object $L$ of $\sC$ is called \emph{invertible} if
the evaluation map $L^\vee \otimes L \to \Lambda$ is an isomorphism,
where $L^\vee$ denotes the (strong) dual of $L$.
It then follows that 
$\End(L) \simeq \Lambda$ 
and hence $L$ is indecomposable
(that is, $\End(L)$ has no projectors other than $0, 1$).
We will use the following result of Krull-Schmidt type.

\begin{ppn}\label{prop:invertible}
Let 
$L_1, \dots, L_n$ be invertible objects of $\sC$,
and put $M:=L_1 \oplus \cdots \oplus L_n$.
Let $N_1, N_2$ be objects of $\sC$
such that there are isomorphisms
$M \simeq N_1 \oplus N_2$ 
and 
$N_1 \simeq L_1 \oplus \cdots \oplus L_r$ for some $1 \le r \le n$.
Then there is an isomorphism
$N_2 \simeq L_{r+1} \oplus \cdots \oplus L_n$.
\end{ppn}
\begin{proof}
This follows from \cite[Chapter 1, Theorem 3.6]{Bass}.
\end{proof}

Artin motives provide basic examples of invertible motives,
as seen in the following lemma.
We shall see more examples in Propositions \ref{as-dual} and \ref {f-mot} below.
 
\begin{lem}\label{artin-motive}
Let $K/\k$ be a finite Galois extension and put $X=\Spec K$, $G=\Gal(X/\Spec \k)$.
For any $\chi_1$, $\chi_2 \in \wh G$, there is an isomorphism
\[h(X)^{\chi_1\chi_2}\simeq h(X)^{\chi_1} \ot h(X)^{\chi_2}\] 
 in $\Chow(\k,\L)$, 
where $\L$ is large enough to contain the values of $\chi_i$'s. 
In particular, $h(X)^\chi$ is invertible for any $\chi \in \wh G$. 
\end{lem}

\begin{proof}Let $\D \colon X \to X \times X$ be the diagonal and consider the morphisms 
\[\D_*\colon h(X) \to h(X\times X), \quad \D^*\colon h(X\times X)\to h(X).\]
Put $d:=|G|$. We shall show that
\begin{align}
\label{eq:isom-artin}
\begin{split}
&d \cdot e^{\chi_1\chi_2} \circ \D^* \circ (e^{\chi_1}\times e^{\chi_2}) 
: h(X)^{\chi_1} \otimes h(X)^{\chi_2} \to h(X)^{\chi_1 \chi_2},
\\
& (e^{\chi_1}\times e^{\chi_2})\circ \D_* \circ e^{\chi_1\chi_2} 
: h(X)^{\chi_1 \chi_2} \to h(X)^{\chi_1} \otimes h(X)^{\chi_2}
\end{split}
\end{align}
are isomorphisms mutually inverse to each other. 
Since 
\[\D^*\circ (g_1,g_2)_* \circ \D_* \circ g_* =((g_1^{-1},g_2^{-1})\circ \D)^* \circ (\D\circ g)_*=
\begin{cases} (g_1g)_* & (g_1=g_2), \\ 0 & (g_1 \ne g_2), \end{cases}\]
for any $g, g_1, g_2 \in G$, we have
\begin{align*}
&\D^* \circ (e^{\chi_1}\times e^{\chi_2}) \circ \D_*\circ e^{\chi_1\chi_2}=
\frac{1}{d^3} \sum_{g,g_1=g_2 \in G} \ol{\chi_1\chi_2}(g)\ol\chi_1(g_1)\ol\chi_2(g_2)(g_1g)_*
\\&=\frac{1}{d^2}\sum_{g'\in G} \ol{\chi_1\chi_2}(g') g'_* =\frac{1}{d} e^{\chi_1\chi_2}. 
\end{align*}
Hence $e^{\chi_1\chi_2} \circ \D^* \circ (e^{\chi_1}\times e^{\chi_2}) \circ \D_*\circ e^{\chi_1\chi_2}
=d^{-1} e^{\chi_1\chi_2}$.

On the other hand, for any $g_1, g_2, g, h_1, h_2 \in G$, the $0$-cycle on $X^4$ 
\begin{align*}
\G(g_1,g_2,g,h_1,h_2):=&(h_1,h_2)_*\circ \D_* \circ g_* \circ \D^*\circ (g_1,g_2)_* 
\\=& ((h_1,h_2)\circ \D\circ g)_*\circ ((g_1^{-1},g_2^{-1})\circ \D)^*\end{align*}
is the image of $X \to X^4$; $x \mapsto (g_1^{-1}x,g_2^{-1}x, h_1gx,h_2gx)$, and we have 
\begin{align*}
& (e^{\chi_1}\times e^{\chi_2}) \circ \D_*\circ e^{\chi_1\chi_2} \circ \D^* \circ (e^{\chi_1}\times e^{\chi_2})
\\&= \frac{1}{d^5} \sum_{g_1,g_2,g,h_1,h_2\in G} \ol\chi_1(h_1gg_1)\ol\chi_2(h_2gg_2)\G(g_1,g_2,g,h_1,h_2). 
\end{align*}
Note that $\G(g_1,g_2,g,h_1,h_2)$ is contained in the graph $(h_1gg_1,h_2gg_2)_*$.
When $g_1$, $g_2$, $g$, $h_1$, $h_2$ range over $G$ with fixed values  $g_1'=h_1gg_1$ and $g_2'=h_2gg_2$, 
the cycles $\G(g_1,g_2,g,h_1,h_2)$ sum up to $d^2 (g_1',g_2')_*$.
Hence we obtain 
\[(e^{\chi_1}\times e^{\chi_2}) \circ \D_*\circ e^{\chi_1\chi_2} \circ \D^* \circ (e^{\chi_1}\times e^{\chi_2})
=d^{-1}(e^{\chi_1}\times e^{\chi_2}).\] 
This proves that \eqref{eq:isom-artin} are isomorphisms.

The last statement follows by letting $\chi_1=\chi$ and $\chi_2=\ol\chi$, since $h(X)^1\simeq \L$ by Proposition \ref{localization} (ii). 
\end{proof}

Let $d$ be a positive integer such that $\mu_d \subset \k$ and assume that $\Q(\zeta_d) \subset \L$. 
For $c \in \k^*$ and $\chi \in\wh\m_d$, let 
\[\mathrm{Km}(c)=\mathrm{Km}_d(c) \colon \Gal(\k(c^{1/d})/\k) \to \m_d; \quad g\mapsto g(c^{1/d})/c^{1/d}\] 
be the Kummer character associated with $c$, and define
\begin{equation}\label{eq:def-Lchi}
\L\langle c \rangle^\chi =h(\Spec \k(c^{1/d}))^{\chi\circ \mathrm{Km}(c)}, 
\end{equation}
an invertible object in $\Chow(\k,\L)$. 

\begin{lem}\label{lem:L-chi}\ 
\begin{enumerate}
\item For $c_1, c_2 \in \k^*$ and $\chi \in \wh\mu_d$,
we have $\Lambda\langle c_1c_2 \rangle^{\chi}\simeq \L\langle c_1 \rangle^{\chi} \otimes \L\langle c_2 \rangle^{\chi}$.
\item For $c \in \k^*$ and $\chi_1, \chi_2 \in \wh\mu_d$,
we have $\Lambda\langle c \rangle^{\chi_1 \chi_2}\simeq \L\langle c \rangle^{\chi_1} \otimes \L\langle c \rangle^{\chi_2}$.
\item
If $\chi$ factors through $\chi' \in \wh\m_{d'}$ for some $d' \mid d$, 
i.e. $\chi(m)=\chi'(m^{d/d'})$ for any $m \in \mu_d$, 
then we have $\L\langle c \rangle^\chi \simeq \L\langle c \rangle^{\chi'}$.
\end{enumerate}
\end{lem}

\begin{proof}
These follow from Proposition \ref{localization} (ii) and Lemma \ref{artin-motive}. 
\end{proof}

%%%
\section{Artin-Schreier and Fermat motives}\label{sect:AS-F}

\subsection{Artin-Schreier motives} \label{s4.1}
In this subsection 
let $\k$ be a finite field with $q$ elements
of characteristic $p$,
and let $d$ be a positive divisor of $q-1$.
Let $A_d^\0$ be the affine {\em Artin-Schreier curve} over $\k$ defined by 
\begin{equation}\label{eq:def-AS}
x^q-x=y^d.
\end{equation}
It admits an action of $\k \times \m_d$ given by $(a,m).(x,y)=(x+a, my)$. 
There is no fixed point of $(a, m) \in \kappa \times \mu_d$ if $a \not= 0$,
while $\m_d$ fixes the points in
$A_d^\0(\k)=\{(x,0) \mid x \in \k\}$.
The projectivization $X^q-XZ^{q-1}=Y^dZ^{q-d}$ is non-singular at the unique point at infinity $[0:1:0]$ if and only if $d=q-1$. 
Let $A_d$ be the projective smooth curve obtained by normalizing the singularity. It has a unique point at infinity, written as $\infty$, 
which we take as the distinguished point.
The action of $\k \times \m_d$ extends to $A_d$
and fixes $\infty$. 

If $d' \mid d$, we have a finite surjective morphism 
\begin{equation}\label{eq:Ad-Ad'}
A_d \to A_{d'}, \qquad (x, y) \mapsto (x, y^{d/d'})
\end{equation}
of degree $d/d'$, compatible with the group actions via the homomorphism 
$\k\times \m_d \to \k \times \m_{d'}, \ (a, m) \mapsto (a, m^{d/d'})$. 
It is generically Galois with the Galois group $\Ker(\m_d \twoheadrightarrow \m_{d'})=\mu_{d/d'}$. 
The genus of $A_d$ is $(q-1)(d-1)/2$, which can be seen by the Riemann-Hurwitz formula for the covering $A_d \to A_1 \simeq \P^1$. 

Let $\L$ be a field containing $\Q(\zeta_{pd})$. 
We have the decomposition in $\Chow(\k,\L)$ 
\[h(A_d) = \bigoplus_{(\psi, \chi) \in \wh\k \times \wh\m_d} h(A_d)^{(\psi,\chi)}, \quad 
h(A_d)^{(\psi,\chi)}:=(A_d,e^{(\psi,\chi)}). \]
Here, $e^{(\psi,\chi)}$ means the algebraic correspondence induced by 
the group-ring element defined in \eqref{eq:def-e-chi}. 
Define a projector $e_\prim \in \Q[\k\times\m_d]$ (with coefficients in $\Q$) by
\begin{equation}
\label{eq:def-e-prim}
e_\prim = \sum_{(\psi,\chi)\in \wh\k\times\wh\m_d, \psi \ne 1, \chi\ne 1} e^{(\psi,\chi)}
=(1-e_\k^1)(1-e_{\m_d}^1).
\end{equation}
Note that, for any projectors $e,f \in\Q[G]$ where $G$ is an abelian group, $1-e$ and $ef$ are also projectors.  
Put $h(A_d)_\prim=(A_d,e_\prim)=\bigoplus_{\psi\ne 1, \chi\ne 1} h(A_d)^{(\psi,\chi)}$. 

\begin{ppn}\label{as-deg}
Suppose that $d' \mid d$ and that $\chi \in \wh\m_d$ factors through $\chi'\in\wh\m_{d'}$ 
(i.e. $\chi(m)=\chi'(m^{d/d'})$). 
Then we have $h(A_d)^{(\psi,\chi)} \simeq h(A_{d'})^{(\psi,\chi')}$.
\end{ppn}
\begin{proof}
This follows from Proposition \ref{localization} (ii) applied to \eqref{eq:Ad-Ad'}.
\end{proof}

\begin{ppn}\label{as-mot}\ 
Let $\psi \in \wh\k$ and $\chi \in \wh\mu_d$.
We use the notations from \eqref{eq:def-h0} and \eqref{eq:def-h1}.
\begin{enumerate}
\item 
If $\psi=1$ and $\chi=1$, then
$h(A_d)^{(1,1)}=h_0(A_d)\oplus h_2(A_d) = \L\oplus \L(1)$. 
\item 
If $\psi\ne1$ or $\chi \ne 1$, then $h(A_d)^{(\psi,\chi)}=h_1(A_d)^{(\psi,\chi)}$. 
\item If only one of $\psi$, $\chi$ is trivial, then $h(A_d)^{(\psi,\chi)}=0$. 
\end{enumerate}
\end{ppn}

\begin{proof}
Proposition \ref{localization} (ii)
applied to $A_d \to A_d/(\k\times\m_d) = \P^1$
yields $h(A_d)^{(1, 1)}=h(\P^1)=\Lambda \oplus \Lambda(1)$,
showing (i).
This also implies $h_i(A_d)=h_i(A_d)^{(1, 1)}$ ($i=0,2$), from which we obtain (ii).
To see (iii), it suffices to apply Proposition \ref{localization} (ii)
to 
$A_d \to A_d/\m_d = \P^1$ (resp. $A_d \to A_d/\k = \P^1$)
when $\chi=1$ (resp. $\psi=1$).
\end{proof}

\begin{lem}\label{intersection}
Let $(\ , \ )\colon \CH_1(A_{q-1} \times A_{q-1})\times \CH_1(A_{q-1} \times A_{q-1}) \to \Z$ be the intersection number pairing. 
For $(a,m)\in \k\times \k^*$, let $\G_{(a,m)} \in \CH_1(A_{q-1} \times A_{q-1})$ be the class of its graph and $\D=\G_{(0,1)}$. Then, 
\[(\D,\G_{(a,m)})=\begin{cases}
3q-q^2 & ((a,m)=(0,1)), \\
q+1 & (a=0, m\ne 1), \\
q & (a \ne 0, m=1), \\
1 & (a \ne 0, m\ne 1). 
\end{cases}\]
\end{lem}
\begin{proof}
First, $(\D,\D)$ equals the Euler-Poincar\'e characteristic of $A_d$, i.e. $2-(q-1)(q-2)$. 
Secondly if $m\ne 1$, then $\D$ and $\G_{(0,m)}$ meet transversally at the $q+1$ points $\{(P,P) \mid P\in A_d(\k)\}$, where 
$A_d(\k)=\{(x,0) \mid x\in\k\}\cup\{\infty\}$. 
Thirdly if $a \ne 0$, then $\D \cap \G_{(a,1)} =\{(\infty,\infty)\}$. 
Since $v_\infty(x)=1-q$, $v_\infty(y)=-q$, 
the completed local ring of $A_d$ at $\infty$ is $\wh\sO_{A_d,\infty}=\k[[t]]$ where $t=x/y$, and $(a,m)$ maps $t$ to $m^{-1}(t+ay^{-1})$. 
We have
\[\k[[t,s]]/(t-s) \ot_{\k[[t,s]]} \k[[t,s]]/(mt-s-ay^{-1})=\k[[t]]/((m-1)t-ay^{-1}). \]
Its length is $q$ if $m=1$, and is $1$ otherwise since $(m-1)t-ay^{-1} \in t \k[[t]]^\times$.  
The proof is complete.
\end{proof}

\begin{lem}\label{intersection-surf}
Let $S$ be a connected smooth projective surface over $\k$,
and take
\[
e, e' \in \CH_1(S)_\Lambda
= \Hom_{\Chow(\k, \L)}(\Lambda(1), h(S))
= \Hom_{\Chow(\k, \L)}(h(S), \Lambda(1)).
\]
Suppose that the intersection number $n:=(e, e')$ is not zero,
and define
\begin{align*}
&\alpha := \frac{1}{n}(e \times e') \in 
\CH_2(S \times S)_\Lambda = \Hom_{\Chow(\k, \L)}(h(S), h(S)).
\end{align*}
\begin{enumerate}
\item 
We have $\alpha^2 = \alpha$ in $\Hom_{\Chow(\k, \L)}(h(S), h(S))$.
\item 
Set $M:=(S, \alpha) \in \Chow(\k, \L)$
and 
let $\ol{\alpha} \colon M \to h(S)$ and $\ul{\alpha} \colon h(S) \to M$ be 
the inclusion and projection 
(so that $\alpha = \ol{\alpha} \circ \ul{\alpha}$).
Then 
\[ 
\ul{\alpha} \circ e \colon \Lambda(1) \to M
\quad \text{and} \quad
\frac{1}{n}e' \circ \ol{\alpha} \colon M \to \Lambda(1)
\]
are isomorphisms mutually inverse to each other.
\end{enumerate}
\end{lem}
\begin{proof}
This is shown by a straightforward computation
using the definition of the composition of correspondences.
\end{proof}

The following are motivic analogues of \eqref{gauss-reflection} and \eqref{gauss-bar}, respectively
(see also Remark \ref{r-inv} below). 

\begin{ppn}\label{as-dual}
Suppose that $\psi \in \wh\k$, $\psi\ne 1$ and $\chi \in \wh\mu_d$, $\chi \ne 1$. 
\begin{enumerate}
\item There is an isomorphism 
\[h(A_d)^{(\psi,\chi)} \ot h(A_d)^{(\ol\psi, \ol\chi)} \simeq \L(1). \]
In particular, $h(A_d)^{(\psi, \chi)}$ is invertible. 
\item 
There is an isomorphism
\[h(A_d)^{(\ol\psi,\chi)} \simeq h(A_d)^{(\psi,\chi)} \ot \L\langle -1 \rangle^\chi, \]
where $\L\langle -1 \rangle^\chi$ is from \eqref{eq:def-Lchi}.
\end{enumerate}
\end{ppn}
\begin{proof}(i) 
By Proposition \ref{as-deg}, we can assume that $d=q-1$. 
Regard $e^{(\psi,\chi)}$ and $e^{(\ol\psi,\ol\chi)}$ as elements of $\CH_1(A_{q-1}\times A_{q-1})_\L$. 
Then the intersection number is computed using Lemma \ref{intersection} as 
\begin{align*}
&(e^{(\psi,\chi)}, e^{(\ol\psi,\ol\chi)})
\\&=\frac{1}{q^2(q-1)^2} \sum_{a,a'\in\k, m,m'\in\k^*} \psi(a'-a) \chi(m'/m) (\G_{(a,m)},\G_{(a',m')})
\\&= \frac{1}{q^2(q-1)^2} \sum_{a,a'\in\k, m,m'\in\k^*} \psi(a'-a) \chi(m'/m) (\D,\G_{(a'-a,m'/m)})
\\&=\frac{1}{q(q-1)} \sum_{a\in\k,m\in\k^*} \psi(a)\chi(m)(\D,\G_{(a,m)})
\\&=\frac{1}{q(q-1)}\left((3q-q^2)+(q+1)\sum_{m\ne 1} \chi(m)+q\sum_{a\ne 0}\psi(a) +\sum_{a\ne 0, m\ne 1}\psi(a)\chi(m) \right)
\\&=\frac{1}{q(q-1)}\left((3q-q^2)-(q+1)-q+(-1)^2\right)=-1. 
\end{align*}
Now Lemma \ref{intersection-surf} completes the proof of (i).

(ii) 
If $p=2$, then the statement is trivial since $\ol\psi=\psi$ and $\L\langle -1 \rangle^\chi=\L$. 
Suppose that $p$ is odd. 
We may further suppose that $d=q-1$ by 
Lemma \ref{lem:L-chi} (iii) and Proposition \ref{as-deg}.
Put $K=\k(\mu_{2(q-1)})$, $A_{q-1,K}=A_{q-1} \times_{\Spec \k} \Spec K$, 
and fix a primitive $2(q-1)$th root of unity $\z \in K$. 
Let $f$ be a $K$-automorphism of $A_{q-1,K}$ defined by $f(x,y)=(-x,\z y)$. 
Then,  we have 
\[ 
f \circ ((a, m) \times \id_{K}) = ((-a, m) \times \id_{K}) \circ f,
\quad
f \circ ((a, m) \times \sigma) = ((-a, -m) \times \sigma)  \circ f
\]
as $\k$-automorphisms of $A_{q-1, K}$,
where $\sigma$ is the generator of $\Gal(K/\k)$.
It follows that 
\[
f_\#\left(e^{(\psi, \chi)} \otimes \frac{1+\sigma}{2}\right) 
= \left(e^{(\ol{\psi}, \chi)} \otimes \frac{1+\chi(-1)\sigma}{2}\right),
\]
where $f_\#$ is from \eqref{eq:f-sharp}.
We conclude
\begin{align*}
h(A_{q-1})^{(\psi, \chi)}
&\simeq
\left(e^{(\psi, \chi)} \otimes \frac{1+\sigma}{2}\right)h(A_{q-1, K})
\\
&
\overset{f}{\simeq}
\left(e^{(\ol\psi, \chi)} \otimes \frac{1+\chi(-1)\sigma}{2}\right)h(A_{q-1, K})
\simeq h(A_{q-1})^{(\ol\psi,\chi)} \ot \L\langle -1 \rangle^\chi,
\end{align*}
as desired.
\end{proof}

%\begin{rmk}\label{rmk:M-Ad-decomp}
%Since $\k \times \m_d$ fixes $A_d\setminus A_d^\0=\{\infty\}$, we have 
%$M(A_d)^{(\psi,\chi)}= M(A_d^\0)^{(\psi,\chi)}$ in $\DM(\k,\L)$ unless $(\psi,\chi)=(1,1)$
%by Proposition \ref{localization} (i),
%and hence
%$M(A_d^\0)_\prim \simeq M(A_d)_\prim$ as well,
%where $(-)_\prim$ denotes the direct factor defined by \eqref{eq:def-e-prim}.
%\end{rmk}

%
\subsection{Fermat motives}\label{s-f-m}
In this subsection 
we let $\k$ be an arbitrary perfect field 
and assume that $d$ is a positive integer 
such that $\mu_d:= \{ m \in \k^* \mid m^d=1 \}$ has $d$ elements.
For each positive integer $n$ and $c \in \k^*$,
let $F_d^{(n)}\langle c \rangle \subset \P^n$ be the (twisted) projective Fermat hypersurface of degree $d$ and dimension $n-1$ defined by 
\[u_1^d+\cdots+u_n^d=cu_0^d. \]
When $c=1$, we just write $F_d^{(n)}$ instead of $F_d^{(n)}\langle 1 \rangle$. 
Let $\m_d^n$ act on $F_d^{(n)}\langle c \rangle$  by 
\[(m_1,\dots, m_n)[u_0:u_1:\dots:u_n]=[u_0:m_1u_1:\dots: m_nu_n].\] 
If $d' \mid d$, we have a finite surjective morphism 
\begin{equation}\label{eq:Fd-Fd'}
F_{d}^{(n)}\langle c \rangle \to F_{d'}^{(n)}\langle c \rangle; \quad [u_0 : \dots : u_n] \mapsto [u_0^{d/d'} : \dots : u_n^{d/d'}],
\end{equation}
compatible with the group actions via the homomorphism 
$\m_d^n \to \m_{d'}^n, \ (m_i) \mapsto (m_i^{d/d'})$. 
It is generically Galois (\'etale over $u_0\cdots u_n \ne 0$) with the Galois group $\Ker(\m_d^n \to \m_{d'}^n)$.

Let $\L$ be a field containing $\Q(\zeta_{d})$.  We have the decomposition in $\Chow(\k,\L)$
\[h(F_d^{(n)}\langle c \rangle) = \bigoplus_{\bchi \in \wh\m_d^n} h(F_d^{(n)}\langle c \rangle)^{\bchi}, \quad 
h(F_d^{(n)}\langle c \rangle)^{\bchi}:=(F_d^{(n)}\langle c \rangle,e^{\bchi}). \]

\begin{ppn}\label{f-deg}
Suppose that $d'\mid d$ and $\bchi \in \wh\m_d^n$ factors through $\bchi'\in\wh\m_{d'}^n$ (i.e. $\bchi(m)=\bchi'(m^{d/{d'}})$).  Then we have $h(F_d^{(n)}\langle c \rangle)^\bchi \simeq h(F_{d'}^{(n)}\langle c \rangle)^{\bchi'}$. 
\end{ppn}
\begin{proof}
This follows from Proposition \ref{localization} (ii) applied to \eqref{eq:Fd-Fd'}.
\end{proof}

Put 
\[\mathfrak{X}_d^{(n)} =\left\{\bchi=(\chi_1,\dots, \chi_n) \in \wh\m_d^n \Bigm| \chi_1,\dots, \chi_n, \prod_{i=1}^n \chi_i \ne 1\right\},\]
and define a projector $e_\prim \in \Q[\m_d^n]$ (with coefficients in $\Q$) by 
\begin{equation}\label{eq:def-prim-fermat}
e_\prim = \sum_{\bchi \in \mathfrak{X}_d^{(n)}} e^\bchi
= \prod_{i=0}^n (1-e_{\io_i(\m_d)}^1), 
\end{equation}
where $\io_i\colon\m_d\to \m_d^n$ is the embedding of the $i$th factor if $i \ne 0$ and $\io_0$ is the diagonal embedding. 
Define 
\[h(F_d^{(n)}\langle c \rangle)_\prim = (F_d^{(n)}\langle c \rangle, e_\prim)=\bigoplus_{\bchi \in \mathfrak{X}_d^{(n)}} h(F_d^{(n)}\langle c \rangle)^\bchi.\] 

\begin{ppn}\label{twist}
For any $\bchi=(\chi_1,\dots, \chi_n) \in \wh\m_d^n$, we have an isomorphism 
\[h(F_d^{(n)}\langle c \rangle)^\bchi \simeq 
h(F_d^{(n)})^\bchi \ot \L\langle c \rangle^{\prod_{i=1}^n\chi_i}.\]
In particular, 
$h(F_d^{(1)} \langle c \rangle)^\chi \simeq \L\langle c \rangle^\chi$
 is invertible for any $\chi\in\wh\m_d$. 
\end{ppn}

\begin{proof}Let $\a$ be a $d$th root of $c$, put $K=\k(\a)$, and consider the $K$-isomorphism 
\[f \colon F_d^{(n)}\langle c \rangle \times \Spec K \to F_d^{(n)} \times \Spec K\]
defined by $f([u_0:u_1:\cdots:u_n])=[u_0: \a^{-1} u_1:\dots : \a^{-1} u_n]$. 
Then $f$ sends the graph of $(\x,g) \in \m_d^n \times \Gal(K/\k)$ to the graph of 
$(\iota_0(\mathrm{Km}(c)(g)^{-1})\x, g)$.
Hence the projector $e^\bchi \times e^1$ is mapped to 
\begin{align*}
&\frac{1}{d^nd} \sum_{\x,g} \ol\bchi(\x) (\iota_0(\mathrm{Km}(c)(g)^{-1})\x, g)
=
\frac{1}{d^nd} \sum_{\x,g} \ol\bchi(\iota_0(\mathrm{Km}(c)(g)\x) (\x, g)
\\&=\frac{1}{d^nd} \sum_{\x,g} \ol\bchi(\x) \left(\prod_{i=1}^n \ol\chi_i(\mathrm{Km}(c)(g))\right) (\x,g) = e^\bchi \times e^{\prod_{i=1}^n \chi_i \circ \mathrm{Km}(c)}. 
\end{align*}
Hence the first assertion follows. 
Since $h(F_d^{(1)}) \simeq \L[\m_d]$, the Artin motive of the regular representation of $\m_d$, we have $h(F_d^{(1)})^\chi \simeq \L$, and the second assertion follows. 
\end{proof}

The following proposition will be generalized in 
Proposition \ref{f-mot} below.

\begin{ppn}\label{f-dual2}
If $\bchi \in \mathfrak{X}_d^{(2)}$,  there is an isomorphism 
\[h(F_d^{(2)}\langle c \rangle)^{\bchi} \ot h(F_d^{(2)}\langle c \rangle)^{\ol\bchi} \simeq \L(1). \]
In particular, $h(F_d^{(2)}\langle c \rangle)^\bchi$ is invertible. 
\end{ppn}
\begin{proof}
This 
can be proved similarly as Proposition \ref{as-dual} (i). 
By Lemma \ref{lem:L-chi} and Propositions \ref{twist}, we can assume that $c=1$. 
For $(m_1,m_2) \in \m_d^2$,  
the intersection numbers on $F_d^{(2)} \times F_d^{(2)}$ are computed as: 
\[(\D,\G_{(m_1,m_2)})=\begin{cases}
2-(d-1)(d-2) & (m_1=m_2=1), 
\\d & (\text{if only one of $m_1$, $m_2$, $m_1m_2^{-1}$ is $1$}), 
\\0 & (\text{otherwise}). 
\end{cases}
\]
Then it follows as before that $(e^{(\chi_1,\chi_2)}, e^{(\ol{\chi_1},\ol{\chi_2})})=-1$, and the proposition follows. 
\end{proof}

To study $h(F_d^{(n)}\langle c \rangle)$ for general $n$, we use the inductive structure of Katsura-Shioda \cite{shioda-katsura}. 
By Proposition \ref{twist}, it suffices to consider the case $c=1$. 
Let $n \ge 2$ and $f_0 \colon F_d^{(n)} \times F_d^{(2)} \dashrightarrow F_d^{(n+1)}$ be 
the rational map defined by 
\begin{align*}
& ([u_0:\cdots:u_{n}], [v_0:v_1:v_2]) 
\\& \mapsto [w_0:\cdots:w_{n+1}]=[u_0v_0:u_1v_0:\cdots:u_{n-1}v_0:u_{n}v_1:u_{n}v_2]. 
\end{align*}
Then $f_0$ is compatible with the actions of $\m_d^n\times\m_d^2$ and $\m_d^{n+1}$ via the map
\begin{equation}\label{eq:def-nu}
\n\colon \m_d^n\times \m_d^2 \to \m_d^{n+1}; \ ((\x_1,\dots, \x_n), (\y_1,\y_2)) \mapsto (\x_1,\dots, \x_{n-1}, \x_n\y_1,\x_n\y_2).
\end{equation}
We have the following commutative diagram:
\[\xymatrix{
F_d^{(n,2)}  \ar[r]^f \ar[d]_\a & F_d^{(n,2)}/H  \ar[d]^\b \\
F_d^{(n)} \times F_d^{(2)}  \ar@{.>}[r]^{f_0}  &  F_d^{(n+1)} .
}\] 
Here,  
\begin{itemize}
\item The morphism $\a$ is the blow-up along $Z=\{u_n=v_0=0\}$. We have 
\begin{equation}\label{eq:def-Z-bu}
Z\simeq F_d^{(n-1)} \times F_d^{(1)}\langle -1 \rangle; 
\quad ([u_i]_{i=0}^n,[v_i]_{i=0}^2) \mapsto ([u_i]_{i=0}^{n-1}, [v_2:v_1]).
\end{equation}
Since $Z$ is smooth, $F_d^{(n,2)}$ is also smooth. 
\item The action of $\m_d^n\times\m_d^2$ on $F_d^{(n)} \times F_d^{(2)}$ respects $Z$ and extends to $F_d^{(n,2)}$. 
\item The morphism $\b$ is the blow-up along $Z_1 \sqcup Z_2$, where $Z_i$ are disjoint smooth closed subschemes of $F_d^{(n+1)}$ defined by
\begin{align*}
& Z_1=\{w_0=\cdots =w_{n-1}=0\}\simeq F_d^{(1)}\langle -1 \rangle; \quad [w_i]_{i=0}^{n+1} \mapsto [w_{n+1}:w_{n}], 
\\& Z_2=\{w_n=w_{n+1}=0\}\simeq F_d^{(n-1)}; \quad [w_i]_{i=0}^{n+1} \mapsto [w_i]_{i=0}^{n-1}.
\end{align*}
\item The action of $\m_d^{n+1}$ on $F_d^{(n+1)}$ respects $Z_i$'s and extends to $F_d^{(n,2)}/H$. 
\item The morphism $f$ is finite and generically Galois with the Galois group
\[H:=\Ker \n=\{((1,\dots,1,\x),(\x^{-1},\x^{-1})) \mid \x \in \m_d\} \simeq \m_d.\]
Also, $f$ is compatible with the group actions via $\n$. 
\end{itemize}

\begin{ppn}\label{hZ}\ 
\begin{enumerate}
\item Let 
$\bchi=(\chi_1,\dots, \chi_{n+1}) \in \wh\m_d^{n+1}$. Then 
\begin{align*}
& h(Z_1)^\bchi \simeq \begin{cases} 
\L\langle -1 \rangle^{\chi_n} 
&  \text{if $\chi_1=\cdots=\chi_{n-1}=\chi_n\chi_{n+1}=1$}, \\ 0 & \text{otherwise}. \end{cases}
\\& h(Z_2)^\bchi \simeq \begin{cases} h(F_d^{(n-1)})^{(\chi_1,\dots, \chi_{n-1})} &  \text{if $\chi_n=\chi_{n+1}=1$}, 
\\ 0 & \text{otherwise}. \end{cases}
\end{align*}

\item Let $\bchi'=((\chi_1,\dots, \chi_n),(\chi_1',\chi_2')) \in \wh\m_d^n \times \wh\m_d^2$. 
Then 
\begin{align*}
& h(F_d^{(n,2)}/H)^{\bchi'}\simeq 
\begin{cases}
h(F_d^{(n,2)})^{\bchi'} & \text{if $\chi_n=\chi_1'\chi_2'$}, 
\\ 0 &  \text{otherwise}. 
\end{cases}
\\& h(Z)^{\bchi'} \simeq 
\begin{cases} h(F_d^{(n-1)})^{(\chi_1,\dots, \chi_{n-1})} \ot \L\langle -1 \rangle^{\chi_1'} & \text{if $\chi_n=\chi_1'\chi_2'=1$}, \\ 0  & \text{otherwise}.\end{cases}
\end{align*} 
\end{enumerate}
\end{ppn}

\begin{proof}
(i) The stabilizer of $Z_1$ (resp. $Z_2$) in $\m_d^{n+1}$ is $\m_d^{n-1}\times \{(\z,\z) \mid \z\in\m_d\}$ (resp. $\{1\}^{n-1} \times \m_d^2$) and $\bchi$ is trivial on this group if and only if $\chi_1=\cdots=\chi_{n-1}=\chi_n\chi_{n+1}=1$ (resp. $\chi_n=\chi_{n+1}=1$), and then 
we have respectively $h(Z_1)^\bchi\simeq \L\langle -1 \rangle^{\chi_n}$ 
by Proposition \ref{twist}, 
and $h(Z_2)^\bchi\simeq h(F_d^{(n-1)})^{(\chi_1,\dots, \chi_{n-1})}$. 
Otherwise, $h(Z_i)^\bchi=0$. 

(ii) The first formula follows from Proposition \ref{localization} (ii) since $(\m_d^n\times \m_d^2)/H \simeq \m_d^{n+1}$ and the pull-back of the characters of $\m_d^{n+1}$ are of the form $((\chi_1,\dots,\chi_{n-1},\chi_n\chi_{n+1}),(\chi_{n},\chi_{n+1}))$. 
On the other hand, the stabilizer of $Z$ in $\m_d^n\times\m_d^2$ is $\{1\}^{n-1} \times \m_d\times \{(\y,\y) \mid \y\in\m_d\}$, and $\bchi'$ is trivial on this subgroup if and only if $\chi_n=1$ and $\chi_1'\chi_2'=1$. 
If the condition is satisfied, we have 
$h(Z)^{\bchi'}\simeq h(F_d^{(n-1)})^{(\chi_1,\dots, \chi_{n-1})} \ot \L\langle -1 \rangle^{\chi_1'}$
by Proposition \ref{twist}, and $h(Z)^{\bchi'}=0$ otherwise. 
Hence the second formula follows. 
\end{proof}

\begin{ppn}\label{f-inductive}
Let $\bchi=(\chi_1,\dots, \chi_{n+1}) \in \wh\m_d^{n+1}$ and put 
\[\bchi^{(n)}=(\chi_1,\dots, \chi_{n-1},\chi_n\chi_{n+1}), \ \bchi^{(2)}=(\chi_n,\chi_{n+1}), 
\ \bchi^{(n,2)}=(\bchi^{(n)},\bchi^{(2)}).\]
Then we have an isomorphism 
\begin{align*}
& h(F_d^{(n+1)})^\bchi \oplus 
\left(\bigoplus_{i=1}^{n-1} h(Z_1)^\bchi(i)\right) \oplus h(Z_2)^\bchi(1) 
\\&  \simeq 
\left(h(F_d^{(n)})^{\bchi^{(n)}} \ot h(F_d^{(2)})^{\bchi^{(2)}}\right) \oplus h(Z)^{\bchi^{(n,2)}}(1).
\end{align*}
\end{ppn}

\begin{proof}
Compute $h(F_d^{(n,2)})^{\bchi^{(n,2)}}$ in two ways using Propositions 
\ref{prop:blow-up} and \ref{hZ}.
\end{proof}

\begin{ppn}\label{f-mot}
Let $n \ge 2$, $\bchi=(\chi_1,\dots, \chi_n) \in \wh\m_d^n$ and 
write $\mathbf{1}=(1,\dots,1)$. 
\begin{enumerate}
\item $h(F_d^{(n)}\langle c \rangle)^{\mathbf{1}} \simeq \bigoplus_{i=0}^{n-1} \L(i)$. 
\item If $\chi_{n-1} \chi_{n}\ne 1$, then 
\[h(F_d^{(n)}\langle c \rangle)^\bchi \simeq h(F_d^{(n-1)}\langle c \rangle)^{(\chi_1,\dots, \chi_{n-2},\chi_{n-1}\chi_n)} \ot h(F_d^{(2)})^{(\chi_{n-1},\chi_n)}.\]
\item  If   $\chi_{n-1} \chi_{n}=1$ and $\bchi\in\mathfrak{X}_d^{(n)}$, then 
\[h(F_d^{(n)}\langle c \rangle)^\bchi \simeq h(F_d^{(n-2)}\langle c \rangle)^{(\chi_1,\dots, \chi_{n-2})} \ot \L\langle -1 \rangle^{\chi_{n-1}}(1).\]
\item If $\bchi \not \in \mathfrak{X}_d^{(n)}\cup\{\mathbf{1}\}$, then $h(F_d^{(n)}\langle c \rangle)^\bchi=0$. 
\end{enumerate}
In particular, 
if $\bchi \in \mathfrak{X}_d^{(n)}$,  
there is an isomorphism 
\begin{equation}\label{j-reflextion-mot}
h(F_d^{(n)}\langle c \rangle)^{\bchi} \ot h(F_d^{(n)}\langle c \rangle)^{\ol\bchi} \simeq \L(n-1), 
\end{equation}
and hence $h(F_d^{(n)}\langle c \rangle)^\bchi$ is invertible. 
\end{ppn}

\begin{proof}
As before, we can assume that $c=1$. 
The case $n=2$ is proved in \cite[Proposition 2.9]{otsubo1}. 
Let $n \ge 2$ and we prove the statements for 
$\bchi=(\chi_1,\dots, \chi_{n+1}) \in \wh\m_d^{n+1}$ by induction on $n$. 
Put $\bchi^{(n-1)}=(\chi_1,\dots,\chi_{n-1})$.
We also use the results and notations of Propositions \ref{hZ} and  \ref{f-inductive}.  
It should be possible to trace the isomorphisms arising from 
these propositions,
but we avoid it by resorting to the Krull-Schmidt principle, 
i.e. Proposition \ref{prop:invertible}.

(1) If $\bchi=\mathbf{1}$, then by the induction hypothesis,  
$h(Z_1)^\bchi \simeq \L$, $h(Z_2)^\bchi
\simeq \bigoplus_{i=0}^{n-2} \L(i)
\simeq h(Z)^{\bchi^{(n,2)}}$, and 
$h(F_d^{(n)})^{\bchi^{(n)}} \ot h(F_d^{(2)})^{\bchi^{(2)}} 
\simeq \left(\bigoplus_{i=0}^{n-1} \L(i)\right) \ot \left(\bigoplus_{i=0}^1 \L(i)\right)$. 
Hence (i) follows. 

From now on, we suppose $\bchi\ne\mathbf{1}$. 

(2) If $\chi_{n-1}\chi_n \ne 1$, then $h(Z_1)^\bchi=h(Z_2)^\bchi=h(Z)^{\bchi^{(n,2)}}=0$, and (ii) follows. 

(2-1) If moreover $\bchi \not \in \mathfrak{X}_d^{(n+1)}$, then we have either $\bchi^{(n)} \not\in \mathfrak{X}_d^{(n)} \cup\{\mathbf{1}\}$ 
or  $\bchi^{(2)} \not\in \mathfrak{X}_d^{(2)} \cup\{\mathbf{1}\}$. Hence $h(F_d^{(n+1)})^\bchi =0$ by (ii) and the induction hypothesis. 

(3) Suppose $\chi_n\chi_{n+1}=1$, so that $\bchi^{(n)} \not\in \mathfrak{X}_d^{(n)}$.  

(3-1) If $\chi_n=\chi_{n+1}=1$, then $\bchi \not\in \mathfrak{X}_d^{(n+1)}$ and $\bchi^{(n)}\ne\mathbf{1}$.  
We have 
$h(Z_1)^\bchi=0$,  
$h(Z_2)^\bchi\simeq (F_d^{(n-1)})^{\bchi^{(n-1)}} \simeq h(Z)^{\bchi^{(n,2)}}$, 
and $h(F_d^{(n)})^{\bchi^{(n)}}=0$ by the induction hypothesis, hence $h(F_d^{(n+1)})^\bchi=0$. 

(3-2) If $\chi_n \ne 1$ (so $\chi_{n+1}\ne 1$), then $h(Z_2)^\bchi=0$, and $h(F_d^{(2)})^{\bchi^{(2)}}=0$ by the induction hypothesis. 

(3-2-1)
If $\bchi^{(n-1)}=\mathbf{1}$ (so $\bchi \not\in\mathfrak{X}_d^{(n+1)}$), then $h(Z_1)^\bchi\simeq\L\langle -1 \rangle^{\chi_n}$ and 
$h(Z)^{\bchi^{(n,2)}} \simeq \bigoplus_{i=0}^{n-2} \L\langle -1 \rangle^{\chi_n}(i)$ by the induction hypothesis, which implies $h(F_d^{(n+1)})^\bchi=0$.  

(3-2-2)
If $\bchi^{(n-1)}\ne \mathbf{1}$, then $h(Z_1)^\bchi=0$. Hence (iii) follows. 
Moreover if $\bchi \not \in \mathfrak{X}_d^{(n+1)}$ (so $n \ge 3$), then $\bchi^{(n-1)} \not \in \mathfrak{X}_d^{(n-1)}$ and 
\[h(F_d^{(n+1)})^\bchi \simeq h(Z)^{\bchi^{(n,2)}}(1)\simeq h(F_d^{(n-1)})^{\bchi^{(n-1)}} \ot \L\langle -1\rangle^{\chi_n}(1)=0\] 
by the induction hypothesis. This finishes the proof of (iv). 
\end{proof}

\begin{rmk}The relations (ii), (iii) are motivic analogues of the functional equations
\begin{align*}
&B(s_1,\dots, s_{n-1},s_n)=B(s_1,\dots, s_{n-2}, s_{n-1}+s_n) B(s_{n-1},s_n), \\
&B(s_1,\dots, s_{n-1},1-s_{n-1})=\frac{B(s_1,\dots, s_{n-2})}{s_1+\cdots+s_{n-2}} \cdot\frac{\pi}{\sin \pi s_{n-1}}, 
\end{align*}
which follows by Remark \ref{gamma-beta}.  
\end{rmk}

\begin{cor}\label{f-mot-Q}
We have 
\[h(F_d^{(n)}\langle c \rangle) \simeq h(F_d^{(n)}\langle c \rangle)_\prim \oplus \bigoplus_{i=0}^{n-1} \Q(i)\] in $\Chow(\k,\Q)$,
where $(-)_\prim$ denotes the direct factor defined by \eqref{eq:def-prim-fermat}.  
\end{cor}
\begin{proof}
This follows from Proposition \ref{f-mot} 
and Lemma \ref{lem:scalar-ext}.
\end{proof}

\begin{rmk}\label{rem:Chow-Kunneth}
Let $L$ be the class of the hyperplane section 
defined by the embedding $\iota\colon F_d^{(n)}\langle c \rangle \hookrightarrow \P^n$,
and define the objects of $\Chow(\k, \Q)$
for $i=0, \dots, 2n-2$ by
\[
h_i(F_d^{(n)}\langle c \rangle):=
\begin{cases}
(F_d^{(n)}\langle c \rangle,  \pi_i) 
& \text{if $i \not= n-1$},
\\
(F_d^{(n)}\langle c \rangle, e_\prim) 
& \text{if $i = n-1$ and $n$ is even},
\\
(F_d^{(n)}\langle c \rangle, \pi_{n-1} + e_\prim) 
& \text{if $i = n-1$ and $n$ is odd},
\end{cases}
\]
where
\[\pi_i := \frac{1}{d} [L^i \times L^{n-i-1}] \in \CH_{n-1}(F_d^{(n)}\langle c \rangle \times F_d^{(n)}\langle c \rangle)_\Q\]
is a projector of $h(F_d^{(n)}\langle c \rangle)$.
We have $(h(F_d^{(n)}\langle c\rangle),\pi_i) \subset h(F_d^{(n)}\langle c \rangle)^{\mathbf{1}}$
since $L$ is fixed by the $\m_d^n$-action.
It follows from Proposition \ref{f-mot} that
$h(F_d^{(n)}\langle c \rangle)=\bigoplus_{i=0}^{2n-2} h_i(F_d^{(n)}\langle c \rangle)$ 
is a Chow-K\"unneth decomposition of $F_d^{(n)}\langle c \rangle$
(\cite[Definition 6.1.1]{murre}).
Note also that, 
if a Weil cohomology theory $H^*$ satisfies the hard Lefschetz,
$e_\prim$ acts on $H^*(F_d^{(n)}\langle c \rangle)$ 
as the projection to the primitive part $\Ker(H^{n-1}(F_d^{(n)}\langle c \rangle) \os{-\cup L}{\lra} H^{n+1}(F_d^{(n)}\langle c \rangle)(1))$,
as is seen from a formula in  \cite[Proposition 1.4.7 (i)]{Kleiman}.
\end{rmk}

\begin{rmk}With no difficulty, we can generalize the results in this subsection to the general diagonal hypersurface
$c_1u_1^d+\cdots +c_n u_n^d=u_0^d$ ($c_i \in\k^*$). We restricted ourselves, however, to the situation as above, which will be needed in subsection \ref{ss7.2}.  
\end{rmk}

%%%
\section{Proof of Theorem \ref{thm1} \rm{(i)}}
In this section we assume $\k$ is a finite field
of characteristic $p$ and of order $q$, and $d$ is a positive divisor of $q-1$.
We will complete the proof of Theorem \ref{thm1} (i).

\subsection{Reduction to the case $n=2$} \label{sect:reduction}

We proceed by induction on $n$. 
The case $n=1$ follows immediately from Proposition \ref{twist}.
The case $n=2$ will be proved in the next subsection. 
Let $n \ge 3$. First assume that $\chi_{n-1}\chi_n\ne 1$,
Then we have by the case $n=2$ and the induction hypothesis 
\begin{align*}
\bigotimes_{i=1}^n h(A_d)^{\chi_i} &\simeq \left(\bigotimes_{i=1}^{n-2} h(A_d)^{\chi_i}\right) \ot h(A_d)^{\chi_{n-1}\chi_n} \ot 
h(F_d^{(2)})^{(\chi_{n-1},\chi_n)}
\\&\simeq h(A_d)^{\prod_{i=1}^n \chi_i} \ot h(F_d^{(n-1)})^{(\chi_1,\dots, \chi_{n-2},\chi_{n-1}\chi_n)} \ot h(F_d^{(2)})^{(\chi_{n-1},\chi_n)}. 
\end{align*}
By Proposition \ref{f-mot} (ii), the formula follows. 
Secondly, assume that $\chi_{n-1}\chi_n=1$.
Then we have by the induction hypothesis and Proposition \ref{as-dual}
\begin{align*}
\bigotimes_{i=1}^n h(A_d)^{\chi_i} &
\simeq 
\left(h(A_d)^{\prod_{i=1}^{n-2} \chi_i}
\ot h(F_d^{(n-2)})^{(\chi_1,\dots, \chi_{n-2})}\right)  \ot \L\langle -1 \rangle^{\chi_{n-1}}(1).
\end{align*}
By Proposition \ref{f-mot} (iii),  the theorem follows. 

\subsection{Proof of the case $n=2$} 

By Propositions \ref{as-deg} and \ref{f-deg}, we can suppose $d=q-1$, so that $\m_d=\k^*$. 
Here we need $\DM(\k,\L)$ from subsection \ref{sect:DM} to treat open varieties.  
We just write $A=A_{q-1}$, $F=F_{q-1}^{(2)}$. 
Let $A^\0 \subset A$ (resp. $F^\0 \subset F$)
be the affine open subscheme defined in \eqref{eq:def-AS}
(resp. by $u_0\ne 0$).
Write $A^\0F=A^\0 \times F$, $A^\0F^\0=A^\0\times F^\0$. 
Define a closed subscheme 
$\G \subset (A^\0)^2 \times A^\0F$ by
\[x_1+x_2=x, \quad u_0y_1=u_1y, \quad u_0y_2=u_2y.\]
Here, the coordinates of the $i$th factor of $(A^\0)^2$ 
are given by $(x_i, y_i)$ subject to the relation $x_i^q-x_i=y_i^{q-1}$.
Those of the first and second factors of $A^\0F$
are $(x, y)$ and $[u_0:u_1:u_2]$,
which are subject to the relations
$x^q-x=y^{q-1}$ and $u_1^{q-1}+u_2^{q-1}=u_0^{q-1}$,
respectively.
Let $\pr_1\colon \G \to (A^\0)^2$ and $\pr_2\colon \G \to A^\0F$ be the projections. 
Put
\[\G_1=\pr_1^{-1}((A^\0)^2\setminus Z), \quad 
\G_2=\pr_2^{-1}(A^\0F^\0),\]
where \[Z
:=\bigsqcup_{a \in \k}
\{ (x_i, y_i)_i \in (A^\0)^2 \mid x_1+x_2=a \} \subset (A^\0)^2.\]
Note that $\G_1 \subset \G$ is defined by $y\ne 0$. 
Since $u_0=0$ implies $u_1, u_2 \ne 0$, hence $y=0$, we have $\G_1 \subset \G_2$. 

Put $G_1=(\k \times \k^*)^2$, $G_2=(\k \times \k^*) \times (\k^*)^2$ and 
$G=(\k \times \k^*)^2 \times \k^*$. 
Let $\pi_1\colon G \to G_1$ be the first projection and define $\pi_2 \colon G \to G_2$ by 
\[\pi_2((a_i,m_i)_i,m)=((a_1+a_2,m), (m_im^{-1})_i). \] 
Then $G$ acts on $(A^\0)^2 \times A^\0F$ via $\pi_1\times \pi_2\colon G \to G_1 \times G_2$ and it respects $\G$, $\G_1$ and $\G_2$. 
Besides, $Z$ is stable under the action of $G_1$ on $A^2$.

\begin{lem}\label{lem:AS-F2}\ 
\begin{enumerate}
\item 
The singular locus of $\G$ is given by $y_1=y_2=y=u_0=0$ (geometrically $q^2(q-1)$ points).
In particular, $\G_1$ and $\G_2$ are non-singular. 
\item 
$\G_1$ is finite Galois over $(A^\0)^2\setminus Z$ with the Galois group $\k^* \subset G$ (embedded as the last component). 
\item
$\G_2$ is finite Galois over $A^\0F^\0$ with the Galois group $\k \subset G$
(embedded as the image of $\k \os{(\mathrm{id}, -\mathrm{id})}\lra \k^2 \subset G$). 
\end{enumerate}
\end{lem}

\begin{proof}
(i) 
This follows from a straightforward computation of the Jacobian matrix.

(ii) 
We have $(A^\0)^2\setminus Z=\Spec R$,
where
\[ R=\k[x_i,y_i,((x_1+x_2)^q-(x_1+x_2))^{-1}]/(x_i^q-x_i-y_i^{q-1}) \]
and $\G_1 =\Spec R[y]/((x_1+x_2)^q-(x_1+x_2)-y^{q-1})$.
Note that on $\G_1$, $x=x_1+x_2$ and $u_i/u_0=y_i/y$. 
Therefore, $\G_1 \to (A^\0)^2\setminus Z$ is the base change by $\Spec R \to \Spec \k[s,s^{-1}]$ ($s=(x_1+x_2)^q-(x_1+x_2)$) 
of $\Spec \k[s,s^{-1},y]/(s-y^{q-1}) \to \Spec \k[s,s^{-1}]$, which has the desired property, and the assertion follows. 

(iii) 
We have $A^\0F^\0=\Spec R$, $R=\k[x,y,t_1,t_2]/(x^q-x-y^{q-1}, t_1^{q-1}+t_2^{q-1}-1)$ ($t_i=u_i/u_0$), and 
$\G_2=\Spec R[x_1]/(x_1^q-x_1-(t_1y)^{q-1})$. 
Note that on $\G_2$, we have $y_i=t_iy$, $x_2=x-x_1$ and $x_2^q-x_2=(x-x_1)^q-(x-x_1)=y^{q-1}-(t_1y)^{q-1}=y_2^{q-1}$. 
Therefore, $\G_2\to A^\0F^\0$ is the base change by  $\Spec R \to \Spec \k[s]$ ($s=(t_1y)^{q-1}$) 
of $\Spec \k[s,x_1]/(x_1^q-x_1-s) \to \Spec \k[s]$, which has the desired property, and the assertion follows. 
\end{proof}

\begin{ppn}
Let $\psi\in\wh\k\setminus\{1\}$, $\chi_1, \chi_2 \in \ck$ and put 
$\bchi=((\psi,\chi_i)_i, 1) \in \wh G = (\wh \k \times \wh \k^*)^2 \times \wh \k^*$.  
\begin{enumerate}
\item If one of $\chi_1, \chi_2$ is non-trivial, then $M(\G_1)^\bchi\simeq M(\G_2)^\bchi$. 
\item If none of $\chi_1, \chi_2, \chi_1\chi_2$ is trivial, then 
$M(\G_1)^\bchi \simeq M(A)^{(\psi,\chi_1)} \otimes M(A)^{(\psi,\chi_2)}$.  
\item 
If $\chi_1\chi_2\ne 1$, then
$M(\G_2)^\bchi \simeq M(A)^{(\psi,\chi_1\chi_2)} \ot M(F)^{(\chi_1,\chi_2)}$. 
\end{enumerate}
\end{ppn}

\begin{proof}
(i) The complement $\G_2\setminus \G_1$ is given by $y=y_1=y_2=0$, on which $(\k^*)^3 \subset G$ acts trivially and we have  $M(\G_2\setminus\G_1)^\bchi=0$. The result follows by Proposition \ref{localization} (i). 

(ii) 
Define $A^* \subset A^\0$  by $y\ne 0$. 
Then $Z^*:=Z \cap (A^*)^2$ is smooth over $\k$. 
First, we have by Proposition \ref{localization} (ii) and Lemma \ref{lem:AS-F2} (ii)
\[M(\G_1)^\bchi \simeq M((A^\0)^2\setminus Z)^{(\psi,\chi_i)_i}=M((A^*)^2\setminus Z^*)^{(\psi,\chi_i)_i}. \]
We will prove $M(Z^*)^{(\psi,\chi_i)_i} =0$. Then it follows by Proposition \ref{localization} (i) that
\[M((A^*)^2\setminus Z^*)^{(\psi,\chi_i)_i}\simeq M((A^*)^2)^{(\psi,\chi_i)_i}
=M(A^*)^{(\psi,\chi_1)} \ot M(A^*)^{(\psi,\chi_2)}. 
\]
Since $A\setminus A^*$ is fixed by the $\k^*$-action and $\chi_i \ne 1$, we have $M(A^*)^{(\psi,\chi_i)} \simeq M(A)^{(\psi,\chi_i)}$ by Proposition \ref{localization} (i), and the assertion follows. 

Now we prove $M(Z^*)^{(\psi,\chi_i)_i} =0$. 
Note that $(x_1+x_2)^q=x_1+x_2$ is equivalent to $y_1^{q-1}+y_2^{q-1}=0$. 
Let $K$ be a quadratic extension of $\k$ and write $X_K$ for $X \times \Spec K$. 
It suffices to prove $M(Z_K^*)^{(\psi,\chi_i)_i} =0$ by Proposition \ref{localization} (ii). 
Choose $\z \in K$ such that $\z^{q-1}=-1$. 
For $(a,m)\in\k\times\k^*$, let $f_{a,m}$ be the $K$-automorphism of $A_K^*$ defined by 
$f_{a,m}(x,y)=(a-x, m\z y)$, and $Z_{a,m}^* \subset (A^*)_K^2$ be its graph, regarded as a $\k$-scheme. 
Then, $Z_K^*=\bigsqcup_{(a,m)\in\k\times\k^*} Z_{a,m}^*$ and 
\begin{equation}\label{eq:Z-star}
M(Z_K^*)^{(\psi,\chi_i)_i}  \simeq \left( \bigoplus_{(a,m)} M(Z_{a,m}^*)\right)^{(\psi,\chi_i)_i} 
 \ \text{in $\DM(\k,\L)$}. 
\end{equation}
By the isomorphism 
\[A_K^*\simeq Z_{a,m}^*; \quad (x,y)\mapsto ((x,y),(f_{a,m}(x,y))),\] 
we have  
$\bigoplus_{(a,m)}M(A_K^*) \simeq M(Z_K^*)$. The action of $\k\times \k^*$ on $A_K^*$ and that of $G_1$ on $Z_K^*$ are compatible under the isomorphism as above and the homomorphism 
\[\d\colon \k\times \k^* \to G_1; \quad (a',m') \mapsto ((a',m'),(-a',m')).\] 
In particular, the action of $\Im(\delta) \subset G_1$ preserves 
the components $Z_{a, m}^*$ of $Z_K^*$.
Since $((\psi, \chi_i)_i) \circ \delta = (1, \chi_1 \chi_2)$ holds in $\wh \k \times \wh \k^*$,
the right hand side of \eqref{eq:Z-star}
is isomorphic to a subobject of $\bigoplus_{(a,m)} M(A_K^*)^{(1, \chi_1\chi_2)}$. 
Since $\chi_1 \chi_2 \not= 1$, we have $M(A_K^*)^{(1, \chi_1\chi_2)} \simeq M(A_K)^{(1, \chi_1\chi_2)}$ as above, and this is trivial by 
Proposition \ref{as-mot} (iii) and \eqref{eq:Chow-DM}.
It follows that $M(Z_K^*)^{(\psi,\chi_i)_i}=0$, as desired.

(iii) We have 
\[M(\G_2)^\bchi \simeq M(A^\0F^\0)^{((\psi,\chi_1\chi_2),(\chi_1,\chi_2))}
=M(A^\0)^{(\psi,\chi_1\chi_2)} \ot M(F^\0)^{(\chi_1,\chi_2)}\] 
by Proposition \ref{localization} (ii)  and Lemma \ref{lem:AS-F2} (iii). 
Since $\k^*$ (resp. the diagonal $\k^* \subset (\k^*)^2$) acts trivially on $A\setminus A^\0$ (resp. $F\setminus F^\0$) and $\chi_1\chi_2\ne 1$, we have $M(A^\0)^{(\psi,\chi_1\chi_2)}\simeq M(A)^{(\psi,\chi_1\chi_2)}$ (resp. $M(F^\0)^{(\chi_1,\chi_2)} \simeq M(F)^{(\chi_1,\chi_2)}$) 
by Proposition \ref{localization} (i), and the result follows.   
\end{proof}

\begin{proof}[Proof of Theorem \ref{thm1} (i)]
We are already reduced to the case $n=2$ in subsection \ref{sect:reduction}.
By the proposition, we have 
\[\bigotimes_{i=1}^2 M(A)^{(\psi,\chi_i)}  \simeq M(A)^{(\psi,\chi_1\chi_2)}\ot M(F)^{(\chi_1,\chi_2)}\]
for $(\chi_1,\chi_2) \in \mathfrak{X}_{q-1}^{(2)}$.
The theorem follows from the full faithfulness of 
\eqref{eq:Chow-DM}.
\end{proof}

%%%
\section{Frobenius endomorphisms}\label{s-fr}

We continue to assume $\k$ is a finite field
of characteristic $p$ and of order $q$, and $d$ is a positive divisor of $q-1$.
The following extends slightly Coleman's result \cite[Theorem A]{coleman}. 
He only considers the Artin-Schreier curves of the form $x^p-x=y^d$ over $\k$, 
so that only Gauss sums with additive characters factoring through the trace $\Tr_{\k/\F_p}$ are involved. 

\begin{ppn}\label{coleman-fr}\ 
\begin{enumerate}
\item If  $\infty \in A_d$ denotes the unique point at infinity,
then we have in $\CH_1(A_d\times A_d)$ 
\begin{align*}
[\Fr_{A_d}]& 
= [g_d] + q[A_d \times \infty] + (2q-1)[\infty \times A_d]. 
\end{align*}
\item 
If we put $Z_d^0=\{u_0=0\} \subset F_d^{(2)}$, then 
we have in $\CH_1(F_d^{(2)} \times F_d^{(2)})$ 
\begin{align*}
[\Fr_{F_d^{(2)}}]&
=[j_d^{(2)}] + 
\frac{q-1}{d}
\left([F_d^{(2)} \times Z_d^0] +2[Z_d^0 \times F_d^{(2)}]\right). 
\end{align*}
\end{enumerate}
\end{ppn}

\begin{proof}
(i) Define $Z=\{y=0\} \subset A_d$
so that $\mathrm{div}(y)=Z-q\cdot \infty$. 
In particular, we have $[Z]=q[\infty]$ in $\CH_0(A_d)$.
Let $((x,y),(x',y'))$ be the coordinates of $A_d \times A_d$, 
and define a function on $A_d \times A_d$ as 
\[f=(x'-x)^\frac{q-1}{d}-\frac{y'}{y}.\]
By abuse of notation, we write the graph of a morphism by the same letter. 
We claim that 
\begin{equation}\label{eq:Coleman-AS}
\mathrm{div}(f) = \Fr_{A_d} +
\sum_{m \in \k^*} (m, m^{\frac{q-1}{d}})
- q(A_d \times \infty)
- (q-1)(\infty \times A_d)
- (Z \times A_d),
\end{equation}
from which the statement follows. 

Write $D$ for the left hand side minus the right hand side of \eqref{eq:Coleman-AS}. 
Using 
$\mathrm{div}(x)=d((0,0)- \infty)$
and
$\mathrm{div}(y)=Z-q\cdot \infty$,
it is straightforward to see that $D$ is effective.
To show $D=0$,
we consider the map
\[ \Pi \colon
\mathrm{Div}(A_d \times A_d) \overset{\pi_*}{\to} 
\mathrm{Div}(\P^1 \times \P^1)
\twoheadrightarrow \CH_1(\P^1 \times \P^1) = \Z \times \Z,
\]
where the first map is the push-forward along the self-product 
$\pi \colon A_d \times A_d \to A_1 \times A_1 \simeq \P^1 \times \P^1$
of \eqref{eq:Ad-Ad'}.
Observe that an effective divisor $E$ on $A_d \times A_d$ 
is trivial if $\Pi(E)=(0,0)$,
since we have a strict inequality $\Pi(C)>(0,0)$ 
for any integral curve $C$ on $A_d \times A_d$.
On the other hand, we compute
\begin{alignat*}{2}
&\Pi(\Fr_{A_d})=(d, qd),& &\Pi(g)=(d, d) \ (g \in \k \times \k^*),
\\
& \Pi(A_d \times \infty)=(d, 0), \quad & &\Pi(P \times A_d)=(0, d) \ (P \in Z \cup \{ \infty \}).
\end{alignat*}
It follows that $\Pi(D)=(0, 0)$.
This completes the proof of (i).

(ii) 
Put $Z_d^i=\{u_i=0\} \subset F_d^{(2)}$
so that $[Z_d^1]=[Z_d^2]=[Z_d^0]$ in $\CH_1(F_d^{(2)})$ since 
$\mathrm{div}(u_i)=Z_d^i-Z_d^0$. 

First, let $d=q-1$ and let $([u_i]_{i=0}^2, [u_i']_{i=0}^2)$ be the coordinates of $F_{q-1}^{(2)} \times F_{q-1}^{(2)}$. 
Define a function on $F_{q-1}^{(2)} \times F_{q-1}^{(2)}$ as 
$$f= 1-\frac{u_0u_1'}{u_0'u_1}-\frac{u_0u_2'}{u_0'u_2}.$$
By a similar argument as (i),
one verifies the equality of divisors
\begin{equation}\label{eq:Coleman-F}
\mathrm{div}(f)=\Fr_{F_{q-1}^{(2)}} - j_{q-1}^{(2)} -F_{q-1}^{(2)} \times Z_{q-1}^0 - Z_{q-1}^1\times F_{q-1}^{(2)}-Z_{q-1}^2 \times F_{q-1}^{(2)},
\end{equation}
from which the statement for $d=q-1$ follows.

For general $d$, 
we look at the image of the both sides of \eqref{eq:Coleman-F}
under the push-forward along
$F_{q-1}^{(2)} \times F_{q-1}^{(2)} \to F_d^{(2)} \times F_d^{(2)}$.
For the left hand side this is
the divisor of the norm of $f$ in $\k(F_d^{(2)} \times F_d^{(2)})$
(hence vanishes in $\CH_1(F_d^{(2)} \times F_d^{(2)})$),
while for the right this is
\[
N^2
\left(\Fr_{F_d^{(2)}} - j_d^{(2)} \right)
-
N^3
\left(
F_d^{(2)} \times Z_d^0+Z_d^1\times F_d^{(2)}+Z_d^2 \times F_d^{(2)}
 \right),
\]
where $N:=(q-1)/d$. We are done.
\end{proof}

\begin{cor}\label{g-j-fr}\ 
We have equalities of endomorphisms of 
$h(A_d)$ (resp. $h(F_d^{(2)})$) 
\[g_d \cdot e_\prim=\Fr_{A_d} \cdot e_\prim
\quad
(\text{resp.} \quad j_d^{(2)} \cdot e_\prim=\Fr_{F_d^{(2)}} \cdot e_\prim)
\]
in $\Chow(\k,\Q)$.
\end{cor}
\begin{proof}
(i) 
We have 
$[A_d \times \infty]\cdot e^{(\psi, \chi)}
=[\infty \times A_d] \cdot e^{(\psi, \chi)} =0$
if none of $\psi \in \wh\k, \chi \in \wh\k^*$ is trivial.
Hence the statement follows from the proposition by Lemma \ref{lem:scalar-ext}.
The proof of (ii) is similar.
\end{proof}

To prove Theorem \ref{thm2}, we prepare computations in group rings. 
Let 
\[\n\colon \Q[\m_d^n \times \m_d^2] \to \Q[\m_d^{n+1}],
\qquad
\rho\colon \Q[\m_d^n\times\m_d^2] \to \Q[\m_d^{n-1}\times \m_d]
\]
be the ring homomorphisms 
induced respectively by \eqref{eq:def-nu} 
and
\[\r((\x_1,\dots,\x_n),(\y_1,\y_2))= ((\x_1,\dots,\x_{n-1}),\y_1/\y_2). \]

\begin{lem}\label{l-j-ind}
For $\bchi=(\chi_1,\dots,\chi_{n+1}) \in \mathfrak{X}_d^{(n+1)}$,
let $\bchi^{(n)}$, $\bchi^{(2)}$ be as in Proposition \ref{f-inductive},
and put $\bchi^{(n-1)}=(\chi_1,\dots,\chi_{n-1})$.
\begin{enumerate}
\item
If $\chi_n\chi_{n+1}\ne 1$, then $\n(j_d^{(n)}e^{\bchi^{(n)}} \times j_d^{(2)}e^{\bchi^{(2)}})=j_d^{(n+1)}e^\bchi$. 
\item There exists an element $j_d^{(n,2)} \in \Q[\m_d^n\times\m_d^2]$ 
independent of $\bchi$,
such that $\n(j_d^{(n,2)})=j_d^{(n+1)}$ and 
\[\r(j_d^{(n,2)}) (e^{\bchi^{(n-1)}} \times e^{\chi_n})
=\chi_n((-1)^{\frac{q-1}{d}})q\left( (j_d^{(n-1)}e^{\bchi^{(n-1)}}) \times e^{\chi_n}\right)\]
unless $\bchi^{(n-1)}=\mathbf{1}$. 
\end{enumerate} 
\end{lem}

\begin{proof}
(i) Put, for $c \in \k$, $S^{(n)}_c=\{(m_i)_i\in(\k^*)^n \mid \sum_{i=1}^n m_i=c\}$, so that 
\[j_d^{(n)}\langle c \rangle =(-1)^{n-1} \sum_{(m_i)_i\in S^{(n)}_c} (m_i^\frac{q-1}{d})_i\]
where 
$j_d^{(n-1)}\langle c \rangle$ is from \eqref{eq:tw-js}.  
Then, $\n$ induces a bijection
$S^{(n)}_1 \times S^{(2)}_1 \to S^{(n+1)}_1\setminus (S^{(n-1)}_1 \times S^{(2)}_0)$. 
By assumption, we have 
\[\sum_{(m_n,m_{n+1}) \in S^{(2)}_0} \chi_n(m_n^\frac{q-1}{d})\chi_{n+1}(m_{n+1}^\frac{q-1}{d})=
\chi_n(-1)^{\frac{q-1}{d}} \sum_{m\in\k^*} \chi_n\chi_{n+1}(m^\frac{q-1}{d})=0. 
\]
Hence 
$\sum_{(m_i)_i\in S^{(n-1)}_1 \times S^{(2)}_0}\bchi((m_i^\frac{q-1}{d})_i)=0$.  
One verifies easily $\n(e^{\bchi^{(n)}} \times e^{\bchi^{(2)}})=e^\bchi$. 
Now it follows 
\begin{align*}
& \n(j_d^{(n)}e^{\bchi^{(n)}} \times j_d^{(2)}e^{\bchi^{(2)}})
=\n(j_d^{(n)} \times j_d^{(2)})e^\bchi
\\&=\left(j_d^{(n+1)}-\sum_{(m_i)_i\in S^{(n-1)}_1 \times S^{(2)}_0} (m_i^\frac{q-1}{d})_i\right)e^\bchi = j_d^{(n+1)}e^\bchi.
\end{align*}

(ii) If we put
\[j_d^{(n,2)}=(-1)^n \sum_{(m_i)_i\in S^{(n+1)}_1}\left((m_1^\frac{q-1}{d},\dots,m_{n-1}^\frac{q-1}{d},1),(m_n^\frac{q-1}{d},m_{n+1}^\frac{q-1}{d})\right),\]
then $\n(j_d^{(n,2)})=j_d^{(n+1)}$. 
On the other hand, 
$\r(j_d^{(n,2)})= \sum_{s,t\in\k,s+t=1} j_d^{(n-1)}\langle s \rangle \times h(t)$,
\begin{align*}
& h(t):=\sum_{(k,l) \in S^{(2)}_t}[(k/l)^\frac{q-1}{d}]=\begin{cases}
(q-1)[(-1)^\frac{q-1}{d}] & \text{if $t=0$}, \\ 
\sum_{m\in\k^*} [m^\frac{q-1}{d}] -[(-1)^\frac{q-1}{d}] & \text{if $t\ne 0$}. 
\end{cases}
\end{align*}
Here we write $[\z] \in\Q[\m_d]$ for the image of $\z\in\m_d$. 
Since $\chi_n\ne 1$, 
\[\sum_{m\in\k^*}[m^\frac{q-1}{d}] e^{\chi_n}=\frac{q-1}{d} \sum_{\z\in\m_d}[\z]e^{\chi_n}=\frac{q-1}{d} \sum_{z\in\m_d}\chi_n(\z) e^{\chi_n}=0.\] 
Since $\bchi^{(n-1)}\ne \mathbf{1}$, we have 
\[\sum_{s \in \k} j_d^{(n-1)}\langle s \rangle e^{\bchi^{(n-1)}}=(-1)^n\sum_{(m_i)_i\in(\k^*)^{n-1}} (m_i^\frac{q-1}{d})_i \cdot e^{\bchi^{(n-1)}}=0.\] 
Therefore, only the term with $(s,t)=(1,0)$ remains and 
\begin{align*}
 \r(j_d^{(n,2)}) (e^{\bchi^{(n-1)}} \times e^{\chi_n})
& = j_d^{(n-1)}e^{\bchi^{(n-1)}} \times q[(-1)^{\frac{q-1}{d}}]e^{\chi_n}
\\& = j_d^{(n-1)}e^{\bchi^{(n-1)}}  \times q\chi_n((-1)^{\frac{q-1}{d}})e^{\chi_n}.
\end{align*}
The proof is complete. 
\end{proof}

\begin{proof}[Proof of Theorem \ref{thm2}]
First, we prove the case $c=1$ by induction on $n$. 
The case $n=2$ follows from Corollary \ref{g-j-fr}.  
It suffices to prove $j_d^{(n+1)}\cdot e^\bchi = \Fr_{F_d^{(n+1)}} \cdot e^\bchi$ 
for each 
 $\bchi=(\chi_1,\dots, \chi_{n+1}) \in \mathfrak{X}_d^{(n+1)}$ ($n\ge 2$). 
First, assume that $\chi_n\chi_{n+1}\ne 1$. Then by Proposition \ref{f-mot} (ii), we have an isomorphism 
\[h(F_d^{(n)})^{\bchi^{(n)}} \ot h(F_d^{(2)})^{\bchi^{(2)}}\os\simeq\lra h(F_d^{(n+1)})^\bchi. \]
Since any morphism in $\Chow(\k, \Lambda)$ commutes with
the Frobenius endomorphisms,
the assertion follows by the induction hypothesis and Lemma \ref{l-j-ind} (i). 
Secondly, assume that $\chi_n\chi_{n+1}=1$ (so $\bchi^{(n-1)}\ne \mathbf{1}$). 
Then by Proposition \ref{f-mot} (iii), we have an isomorphism
\[h(F_d^{(n-1)})^{\bchi^{(n-1)}} \ot \L\langle -1 \rangle^{\chi_n}(1) \os\simeq\lra h(F_d^{(n+1)})^\bchi,\]
Here we used the identification $Z \simeq F_d^{(n-1)} \times F_d^{(1)}\langle -1 \rangle$ 
given by \eqref{eq:def-Z-bu}, under which 
the actions of $\m_d^n \times \m_d^2$ and $\m_d^{n-1} \times\m_d$ are compatible via the homomorphism $\rho$. 
Since the Frobenius acts on $\L\langle -1 \rangle^{\chi_n}(1)$ as the multiplication by $\chi_n \circ \mathrm{Km}(-1)q$ by definition, 
the assertion follows by the induction hypothesis and Lemma \ref{l-j-ind} (ii). 

The general case is reduced to the previous case using Proposition \ref{twist}. 
Note that the Frobenius acts on $\L\langle c \rangle^{\prod_{i=1}^n\chi_i}$ as the multiplication by 
\[\prod_{i=1}^n \chi_i(\mathrm{Km}(c)(\Fr))=\prod_{i=1}^n \chi_i(c^{\frac{q-1}{d}})=\bchi(c^{\frac{q-1}{d}},\dots, c^{\frac{q-1}{d}}),\] 
and that $j_d^{(n)}\langle c \rangle = (c^{\frac{q-1}{d}},\dots, c^{\frac{q-1}{d}})j_d^{(n)}$. 
Hence the proof of Theorem \ref{thm2} is complete. 
\end{proof}

The following corollary follows immediately by Lemma \ref{lem:scalar-ext}. 
\begin{cor}\label{cor:thm2}
For any $c\in \k^*$, we have 
\[j_d^{(n)}\langle c\rangle \cdot e_\prim = \Fr_{F_d^{(n)}\langle c \rangle} \cdot e_\prim\]
in $\End_{\Chow(\k, \Q)}(h(F_d^{(n)}\langle c \rangle)$. 
\end{cor}

\begin{cor}\label{cor:motivic-GJ}
Under the isomorphism of Theorem \ref{thm1} (i), we have for $(\chi_1,\dots, \chi_n)\in\mathfrak{X}_d^{(n)}$
\[ \bigotimes_{i=1}^n \left(g_d \cdot e^{(\psi,\chi_i)}\right) = \left(g_d \cdot e^{(\psi, \prod_{i=1}^n \chi_i)} \right) \otimes \left( j_d^{(n)} \cdot e^{(\chi_1,\dots, \chi_n)} \right)
\]
in $\End(h(A_d)^{\otimes n})=\End(h(A_d) \otimes h(F_d^{(n)}))$. 
\end{cor}
\begin{proof}
Since the isomorphism of Theorem \ref{thm1} (i) is compatible with the Frobenius endomorphisms, the statement follows by Theorem \ref{thm2} and Corollary \ref{g-j-fr}. 
\end{proof}

\begin{rmk}\label{r-inv}
Proposition \ref{coleman-fr} (i) shows that 
the Frobenius acts on $h(A_d)_\prim$ as $g_d$,
and hence on $h(A_d)^{(\psi,\chi)}$ by the multiplication by $g(\psi,\chi)$
if $\psi \in \wh\k$ and $\chi \in \wh\m_d$ are non-trivial.
From this fact, 
together with Proposition \ref{as-dual},
we can deduce \eqref{gauss-reflection} and \eqref{gauss-bar}.
Similarly, 
we can deduce \eqref{e1}, \eqref{e1-reflextion}, and
for $(\chi_1, \dots, \chi_n) \in \mathfrak{X}_d^{(n)}$
\begin{equation}
\label{eq:j-induction}
j(\chi_1, \dots, \chi_n) =
\begin{cases}
j(\chi_1, \dots, \chi_{n-2}, \chi_{n-1}\chi_n)
j(\chi_{n-1},\chi_n) 
& \text{if} \ \chi_{n-1} \chi_n \not= 1,\\
j(\chi_1, \dots, \chi_{n-2})\chi_{n-1}((-1)^{\frac{q-1}{d}})q
& \text{if} \ \chi_{n-1} \chi_n = 1\\
\end{cases}
\end{equation}
from Corollary \ref{cor:motivic-GJ}, \eqref{j-reflextion-mot},
and Proposition \ref{f-mot} (ii), (iii), respectively.

The invertibility (Proposition \ref{f-mot}) implies that, 
for any Weil cohomology theory $H$ whose coefficient field contains $\mu_d$, 
the $\bchi$-eigenspace $H(F_d^{(n)})^{\bchi}$ for $\bchi \in \mathfrak{X}_d^{(n)}$
is one-dimensional
and
the Frobenius acts on this space as the multiplication by $j(\bchi)$
by Theorem \ref{thm2}.
Classically, 
this fact is proved by point counting and the Grothendieck-Lefschetz trace formula.
We have given a purely motivic proof that avoids the use of Weil cohomology.
A similar remark applies to $H(A_d)^{(\psi,\chi)}$.
\end{rmk}

\begin{rmk}If $\psi, \psi' \in \wh\k$ are nontrivial, there exists a unique $c\in\k^*$ such that $\psi'(x)=\psi(cx)$ for any $x\in \k$. This follows from the non-degeneracy of the paring $\k^2 \to \F_p$; $(c,x)\mapsto \Tr_{\k/\F_p}(cx)$. One shows easily 
\begin{equation}\label{psi-psi'}
g(\psi',\chi)=\ol\chi(c^\frac{q-1}{d}) g(\psi,\chi).
\end{equation}
This identity is interpreted motivically as follows. Let $A_d\langle c\rangle$ be defined by $x^q-x=cy^d$. Let $\k\times\k^*$ act on this as before and $h(A_d\langle c \rangle)^{(\psi,\chi)}$ be the associated motive (see subsection \ref{s4.1}).  Then the isomorphism $A_d \to A_d\langle c \rangle$; $(x,y)\mapsto (cx,y)$ induces an isomorphism
\[h(A_d)^{(\psi',\chi)} \simeq h(A_d\langle c \rangle)^{(\psi,\chi)}\] 
of motives. The Gauss sum element $g_d$ acting on $h(A_d)$ corresponds to the element 
\[-\sum_{m\in\k^*} (cm, m^\frac{q-1}{d})
= -\sum_{m\in\k^*} (m, (c^{-1}m)^\frac{q-1}{d})\]
acting on $h(A_d\langle c \rangle)$, hence \eqref{psi-psi'} follows by the argument in Remark \ref{r-inv}. 
\end{rmk}

%%%
\section{Davenport-Hasse relations}

\subsection{Base-change formula}

Let $\k$ be a finite field of characteristic $p$ and of order $q$ and $K$ be a degree $r$ extension of $\k$.  
To avoid the confusion, we write the Gauss and Jacobi sums over $\k$ as $g_\k(\psi,\chi)$ and $j_\k(\chi_1,\dots, \chi_n)$, and similarly over $K$. 
For $\psi \in\wh\k$, let 
\[\psi_K=\psi\circ \Tr_{K/\k} \in \wh K\] be the lifted character. 
For $\psi \in \wh\k$, $\psi \ne 1$ and $\chi, \chi_1,\dots, \chi_n \in \wh\m_d$, $(\chi_1,\dots, \chi_n)\ne(1,\dots, 1)$, 
we have the classical Davenport-Hasse base-change formulas \cite{d-h}
\begin{align}
&g_K(\psi_K,\chi)=g_\k(\psi,\chi)^r, \label{g-dh}
\\&j_K(\chi_1,\dots,\chi_n)=j_\k(\chi_1,\dots, \chi_n)^r.\label{j-dh}
\end{align}
The latter follows from the former by \eqref{e1}. 
See \cite{weil} for an elementary proof and \cite[Sommes trig., 4.12]{deligne} for an $l$-adic sheaf-theoretic proof. 

We give their motivic proofs. 
Let  $\lambda\colon\Spec K\to\Spec\k$ be the structure morphism. 
Then $\lambda$ induces a functor \[\lambda^*\colon \Chow(\k,\L) \to \Chow(K,\L)\] such that 
$\lambda^*(h(X))=h(X_K)$, where $X_K:=X\times_{\Spec \k}\Spec K$. 

Define $g_{d,K/\k} \in \End(h(A_{d,K}))$ by the element
\[- \sum_{m \in K^*} (\Tr_{K/\k}(m), m^{\frac{q^r-1}{d}}) \ \in \Q[\k \times \m_d],\] 
and for $c \in K^*$, $j_{d,K}^{(n)}\langle c \rangle \in \End(h(F_{d,K}^{(n)}\langle c \rangle)$ by the element (as before)
\[(-1)^{n-1}\sum_{m_1,\dots, m_n\in K^*, \sum_{i=1}^nm_i=c} (m_1^\frac{q^r-1}{d},\dots, m_n^\frac{q^r-1}{d}) \in \Q[\m_d^n].\]
Let $e_\prim \in \End(h(A_{d,K}))$ (resp. $e_\prim \in \End(h(F_{d,K}^{(n)}\langle c \rangle)$) 
be the projector onto the primitive part with respect to the $\k \times \m_d$-action (resp. $\m_d^n$-action). 

\begin{thm}\label{dh-1}\
We have equalities of endomorphisms of 
$h(A_{d,K})$ (resp. $h(F_{d,K}^{(n)}\langle c \rangle)$)
\[
\lambda^*(g_{d})^r \cdot e_\prim=g_{d,K/\k} \cdot e_\prim
\quad
(\text{resp.} \quad \lambda^*(j_{d}^{(n)}\langle c \rangle)^r \cdot e_\prim=j_{d,K}^{(n)}\langle c \rangle \cdot e_\prim)
\]
in $\Chow(K,\Q)$.
\end{thm}

\begin{proof}
Let $A'$ be the Artin-Schreier curve over $K$ of degree $q^r$ with the affine equation $u^{q^r}-u =v^d$. 
Let $f \colon A' \to A_{d,K}$ be the morphism defined by $x=\sum_{i=0}^{r-1} u^{q^i}$, $y=v$. It is finite of degree $q^{r-1}$. 
Let 
\[g'=-\sum_{m \in K^*} (m, m^\frac{q^r-1}{d}), \quad e_\prim' \ \in \End_{\Chow(K,\Q)}(h(A'))\]
 be the Gauss sum element and the projector to the primitive part (with respect to the $K\times\m_d$-action), respectively. 
Since $f$ is compatible with the group actions via the homomorphism $\Tr_{K/\k} \times \mathrm{id} \colon K \times \m_d \to \k \times \m_d$, 
we have $f_\#(g')=g_{d,K/\k}$. 
Since $\Tr_{K/\k}$ is a surjective homomorphism, the map $\Q[K] \to \Q[\k]$ maps $e_K^1$ to $e_\k^1$, and we have 
$f_\#(e_\prim')=e_\prim$. 
We have also $f_\#(\Fr_{A'})=\Fr_{A_{d,K}}$. 
Since $g' \cdot e_\prim' = \Fr_{A'} \cdot e_\prim'$ by Corollary \ref{g-j-fr}, we obtain 
$g_{d,K/\k}\cdot e_\prim = \Fr_{A_{d,K}} \cdot e_\prim$. 
On the other hand, we have 
\[\lambda^*(g_{d})^r \cdot e_\prim=\lambda^*(g_d^r \cdot e_\prim)= \lambda^*(\Fr_{A_d}^r \cdot e_\prim)
=\Fr_{A_{d,K}} \cdot e_\prim\]
by Corollary \ref{g-j-fr}. 
Note that the action of $\Q[\k\times\m_d]$, in particular of $e_\prim$, is compatible with $\lambda^*$ and commute with the Frobenius. 
Hence the first assertion is proved. 
The second assertion follows similarly (and more easily) by Corollary \ref{cor:thm2}. 
\end{proof}

We can recover the classical formulas \eqref{g-dh}, \eqref{j-dh} from the theorem 
by multiplying the both sides  by $e^{(\psi,\chi)}$ (resp. by $e^{(\chi_1,\dots,\chi_n)}$), 
together with the fact that $h(A_{q-1,K})^{(\psi,\chi)}=\lambda^*(h(A_{q-1})^{(\psi,\chi)})$ (resp. $h(F_{q-1,K}^{(n)})^{(\chi_1,\dots, \chi_n)}=\lambda^*(h(F_d^{(n)})^{(\chi_1,\dots, \chi_n)})$) is invertible. 

\subsection{Multiplication formula}\label{ss7.2}

Recall the multiplication formula for the gamma function
\begin{equation}\label{gamma-mult}
\frac{\G(ns)}{\G(n)}=n^{n(s-1)} \prod_{i=0}^{n-1} \frac{\G(s+\frac{i}{n})}{\G(1+\frac{i}{n})}.
\end{equation}
Its finite analogue is due to Davenport-Hasse \cite{d-h}. 
Let $\k$ be as in the preceding subsection and suppose $n \mid d \mid q-1$. Then we have 
\begin{equation}\label{eq:mult-g}
g(\psi,\a^n)= \a^n(n)\prod_{\chi^n=1} \frac{g(\psi,\a\chi)}{g(\psi,\chi)}
\end{equation}
for any $\a \in \wh\m_d$. The statement is evident if $\a^n=1$, so we assume $\a^n\ne 1$. Then it is equivalent by \eqref{e1} to 
\begin{equation}\label{mult-j}
\a^n(n) j(\underbrace{\a,\dots, \a}_\text{$n$ times})=\prod_{\chi^n=1,\chi\ne 1} j(\a,\chi).
\end{equation}

Its first elementary proof which does not use Stickelberger's theorem on the prime decomposition of Gauss sums (see \eqref{stickelberger} below) or cohomology was given recently in \cite[Appendix A]{otsubo}. 
It applies a point counting argument based on the geometric construction of Terasoma \cite{terasoma} in his cohomological proof. 

We prove the following motivic analogue,
which includes Theorem \ref{thm1} (ii).

\begin{thm}
Let $\k$ be a field and assume that $d$ is a positive integer such that $\mu_d:= \{ m \in \k^* \mid m^d=1 \}$ has $d$ elements and $\Q(\zeta_{d}) \subset \L$.
For any $n \mid d$ and $\a \in \wh\m_d$ such that $\a^n\ne 1$, there is an isomorphism in $\Chow(\k,\L)$
\begin{equation}\label{mot-mult}
h(F_{d}^{(n)}\langle n\rangle)^{(\a,\dots,\a)} \simeq \bigotimes_{\chi\in\wh\m_d, \chi^n=1, \chi\ne 1} h(F_{d}^{(2)})^{(\a,\chi)}. 
\end{equation}
\end{thm}

\begin{proof}
We may assume $\k$ is a finite field or $\k=\Q(\zeta_d)$.
(In particular, $\k$ is perfect.)
We use the following notations after \cite[Appendix A]{otsubo}: 
\begin{itemize}
\item $S \subset \A^n$ is a hyperplane defined by $s_1+\cdots+ s_n=n$, $s_1\cdots s_n \ne 0$. 
\item $T \to S$ is a covering defined by $t^{d}=s_1\cdots s_n$. It is Galois with the natural identification $\Gal(T/S) =\m_{d}$. 
\item $X \subset \A^n$ is a hypersurface defined by 
$t_1^{d}+\cdots + t_n^{d}=n$, $t_1\cdots t_n\ne 0$. 
There is a morphism 
\[X \to T; \quad s_i=t_i^{d}, \ t=t_1\cdots t_n.\] 
Then $X$ is Galois over $S$ with the natural identification $\Gal(X/S)=\m_{d}^n$, 
under which $\Gal(X/S)\to\Gal(T/S)$ sends $(\x_1,\dots, \x_n)$ to $\prod_{i=1}^n \x_i$. 
\item $C$ is an affine curve defined by $x^{d}+y^n=1$, $x \ne 0$. 
There is a morphism
\[C^{n-1} \to T; \quad s_i= \prod_{j=1}^{n-1} (1-\z^i y_j), \ t=x_1\cdots x_{n-1}, \]
where 
$(x_j, y_j)$ is the coordinate of the $j$th component of $C^{n-1}$, and
$\z\in\k$ is a fixed primitive $n$th root of unity. 
Then $C^{n-1}$ is generically Galois over $S$ with the natural identification 
$\Gal(C^{n-1}/S)=\m_{d}^{n-1} \rtimes S_{n-1}$, 
under which $\Gal(C^{n-1}/S)\to\Gal(T/S)$ sends $((\x_1,\dots, \x_{n-1}), \s)$ to $\prod_{i=1}^{n-1} \x_i$. 
\end{itemize}

First, $X$ is an open subscheme of $F_{d}^{(n)}\langle n\rangle$ and we have isomorphisms 
\[M(T)^\a \simeq M(X)^{(\a,\dots, \a)} \simeq  M(F_{d}^{(n)}\langle n\rangle)^{(\a,\dots,\a)}\] 
by Proposition \ref{localization}.  
For the latter, note that the diagonal in $\m_{d}^n$ acts trivially on $F_{d}^{(n)}\langle n\rangle \setminus X$ but $(\a,\dots, \a)$ is non-trivial on the diagonal since $\a^n\ne 1$. 
In particular, $M(T)^\a$ is invertible by Proposition \ref{f-mot} (iv). 

On the other hand, if $\b \in \wh{\m_{d}^{n-1} \rtimes S_{n-1}}$ denotes the pull-back of $\a$, then 
\[M(T)^\a \simeq M(C^{n-1})^\b\]
by Proposition \ref{localization} (ii). The restriction of $\b$ to the first  (resp. the second) component is $(\a,\dots, \a)$ (resp. $1$), 
and the corresponding projectors $e^{(\a,\dots \a)}$ and $e^1$ satisfy
\[e^\b=e^{(\a,\dots,\a)} e^1 = e^1 e^{(\a,\dots, \a)} \quad \text{in} \ \L[\m_{d}^{n-1} \rtimes S_{n-1}] . \]
Hence $e^1$ restricts to an idempotent endomorphism of $M:=M(C^{n-1})^{(\a,\dots, \a)}$. 
Since $C$ admits an action of $\m_{d} \times \m_n$, we can further decompose $M$ with respect to the $\m_n^{n-1}$-action. 
Let $\n_1,\dots, \n_{n-1} \in \wh\m_n\setminus\{1\}$ be distinct characters (such a choice is unique up to permutations), 
and $e^\bn$ be the corresponding projector. 
One easily verifies that $\s e^\bn=e^{\s\bn}\s$ for any $\s \in S_{n-1}$, where $S_{n-1}$ acts on $\wh\m_n^{n-1}$ as permutations. 
Hence 
\[e^\bn e^1 e^\bn = \frac{1}{(n-1)!} \sum_{\s} e^\bn e^{\s\bn} \s = \frac{1}{(n-1)!} e^\bn \quad \text{in}  \ \End(M). \]
Note that $e^\bn e^{\s\bn}=0$ unless $\s=1$ by the assumption on $\bn$. 
If we put $L=M^{\bn}$ and $N=M^{1} \simeq M(T)^\a$, then the composite $L \to N \to L$ of the natural morphisms (via $M$) is the multiplication by $1/(n-1)!$, hence is an isomorphism. Since $N$ is invertible, it follows that $L \simeq N$ once we show that $L$ is also invertible. 
We have by definition $L=\bigotimes_{i=1}^{n-1} M(C)^{(\a,\n_i)}$. 
If $\wt C \subset F_{d}^{(2)}$ denotes the open curve defined by $u_0u_1\ne 0$, $C$ is the quotient of $\wt C$ by $1 \times \m_{d/n} \subset \m_{d}^2$, and we have isomorphisms by Propositions \ref{localization} 
\[M(C)^{(\a,\n_i)} \simeq M(\wt C)^{(\a,\chi_i)}\simeq M(F_{d}^{(2)})^{(\a,\chi_i)},\]
where $\chi_i$ is the pull-back of $\n_i$ to $\m_{d}$. 
For the latter isomorphism, note that the diagonal in $\m_{d}^2$ (resp. $\m_{d} \times 1$) acts trivially on $\{u_0=0\}$ (resp. $\{u_1=0\}$) and $\a\chi_i$ (resp. $\a$) is non-trivial on the subgroup. 
Since $M(F_{d}^{(2)})^{(\a,\chi_i)}$ is invertible by Proposition \ref{f-mot}, $L$ is also invertible and we have 
$\bigotimes_{i=1}^{n-1} M(F_{d}^{(2)})^{(\a,\chi_i)} \simeq N \simeq M(F_{d}^{(n)}\langle n\rangle)^{(\a,\dots,\a)}$, 
which implies the desired isomorphism of Chow motives. 
\end{proof}

\begin{cor}\label{cor:DH2}
Suppose further that $\k$ is a finite field. Under the isomorphism \eqref{mot-mult}, the endomorphism $j_{d}^{(n)}\langle n\rangle$ on the left hand side corresponds to $\bigotimes_{\chi^n=1,\chi\ne 1} j_{d}^{(2)}$ on the right hand side. 
\end{cor}

\begin{proof}
This follows by comparing the Frobenius endomorphisms using Theorem \ref{thm2}. 
\end{proof}

%\begin{cor}
%Under the isomorphism \eqref{mot-mult}, the endomorphism $j_{d}^{(n)}\langle n\rangle$ on the left hand side corresponds to $\bigotimes_{\chi^n=1,\chi\ne 1} j_{d}^{(2)}$ on the right hand side. 
%\end{cor}
%
%\begin{proof}
%As before, we may assume $\k$ is a finite field or $\Q(\zeta_d)$.
%When $\k$ is finite, the corollary follows 
%by comparing the Frobenius endomorphisms using Theorem \ref{thm2}. 
%Suppose now $\k=\Q(\zeta_d)$. 
%We take a prime $p$ such that $p \equiv 1 \bmod d$,
%...
%\end{proof}

\begin{rmk}\label{rem:deduce-formulas}
One can deduce \eqref{e0-1} and \eqref{e0-3}
from Coleman's theorem \eqref{eq:coleman}
and Theorems \ref{thm1}, \ref{thm2} as follows.
Since both sides of Theorem \ref{thm1} (i) are invertible,
their endomorphism rings are canonically isomorphic to $\Lambda$.
The Frobenius endomorphisms on both sides,
regarded as elements of $\Lambda$, yield the same element
because the Frobenius endomorphism commutes with any morphism
(see \eqref{eq:Fr-comm-any-f}).
We conclude \eqref{e0-1} by observing 
that the Frobenius endomorphisms agree
with the left and right hand sides of \eqref{e0-1} 
by \eqref{eq:coleman} and Theorem \ref{thm2}, respectively.

Similarly, one recovers 
\eqref{e0-3} (which is \eqref{mult-j}) from Corollary \ref{cor:DH2}
and the invertibility of the motives, noting 
\[j_{d}^{(n)}\langle n\rangle \cdot e^{(\a,\dots, \a)}= \a^n(n)j(\a,\dots,\a) \cdot e^{(\a,\dots, \a)}, 
\quad j_{d}^{(2)} \cdot e^{(\a,\chi)} = j(\a,\chi) \cdot e^{(\a,\chi)}.\] 
\end{rmk}

\begin{rmk}\label{rem-gamma-mult}
If $\k=\C$, the complex period of an invertible motive in $\Chow(\C,\Q(\zeta_d))$ is an element of $\C^*/\Q(\zeta_d)^*$, 
defined by the de Rham-Betti comparison isomorphism. 
Since the periods of Fermat motives are special values of the beta function, the 
 isomorphism \eqref{mot-mult} implies \eqref{gamma-mult} 
for any $s \in d^{-1}\Z$, up to $\Q(\zeta_d)^*$. 
\end{rmk}

%%%
\section{Weil numbers}\label{sect:weil}

In this section, we start with the cyclotomic field $F=\Q(\z_d)$ for an integer $d\ge 3$. 
We have an isomorphism 
\[(\Z/d\Z)^* \simeq G:=\Gal(F/\Q); \quad h \mapsto \s_h,\]
where $\s_h$ is defined by $\s_h(\z_d)=\z_d^h$. Note that $\s_{-1}$ agrees with the complex conjugation (for any embedding $F\hookrightarrow \C$). 
Let $\m(F)$ denote the group of all the roots of unity in $F$. 
Note that $|\m(F)|=d$ or $2d$ according to the parity of $d$. 
Let $v$ be a prime of $F$ over a rational prime $p\nmid d$, $\k$ be the residue field at $v$ and put $q=p^f=|\k|$. 
Let $D=\langle \s_p \rangle \subset G$ be the decomposition subgroup of $v$. 
Then
$G/D$ is bijective to the set of primes of $F$ over $p$ by $\s \mapsto \s v$. 
We have $|D|=f$ and $|G/D|=\vp(d)/f$, where $\vp$ denotes Euler's totient function.

Let $\chi_d \colon \k^* \to F^*$ be the $d$th power residue character modulo $v$, i.e. 
$\chi_d(x \bmod v) \equiv x^\frac{q-1}{d} \pmod{v}$ for any $x\in F^*$ such that $v(x)=0$.
For any $\ba=(a_1, \dots, a_n) \in (\Z/d\Z)^n$ ($n \ge 2$),  put
\[j_d(\ba) = j(\chi_d^{a_1}, \dots, \chi_d^{a_n}) \ \in F^*.\] 
Define 
\[A_d^{(n)}=\{\ba=(a_1,\dots, a_n)  \in (\Z/d\Z)^n \mid a_0, a_1,\dots, a_n \ne 0\}\] 
where $a_0:=-\sum_{i=1}^n a_i$. 
If $\ba \in A_d^{(n)}$, then 
$j_d(\ba)$ is a $q$-Weil number of weight $n-1$ by \eqref{e1-reflextion}.   
Let 
\[W=W_q(F) \subset F^*\] 
be the subgroup of $q$-Weil numbers and 
$J\subset W$ be the subgroup generated by $\{j_d(\ba)\mid \ba \in A_d^{(n)}, n \ge 2\}$. 
Since $\s_h j_d(\ba)= j_d(h\ba)$, $J$ is $G$-stable as well as $W$. 
In fact, $J$ is generated by $\{j_d(\ba)\mid \ba \in A_d^{(2)}\} \cup \{ (-1)^{\frac{q-1}{d}} \}$ 
by \eqref{eq:j-induction}.  

Define the group homomorphism 
\[\Phi \colon W \to \Z[G/D] \]
by $\Phi(\a) v = (\a)$ (the principal divisor of $\a$). 
(Note that we have $v'(\alpha)=0$ for any finite place $v' \nmid p$ and $\alpha \in W$.)
Then for $\a$ of weight $w$, we have 
\begin{equation}\label{eq:8-1}
\Phi(\a)+ \s_{-1} \Phi(\a)
=\Phi(\a \cdot\s_{-1}(\a))
=\Phi(q^w)
=wfT,
\end{equation}
where $T:=\sum_{\t \in G/D} \t$ is the trace element. 
Kronecker's theorem \cite{Kronecker} shows
\begin{equation}\label{e8.0}
\Ker\Phi = \m(F).
\end{equation}
In particular, $W$ is a finitely generated abelian group whose torsion part is precisely $\mu(F)$
(see \eqref{eq:rank} below for its rank).
Since $W$ and $J$ are $G$-stable and $\Phi$ is $G$-equivariant, $\Phi(W)$ and $\Phi(J)$ are ideals of $\Z[G/D]$. 

The case when $f$ is even is easy. 
\begin{ppn}\label{f even} If $f$ is even, then 
$\Phi(W)=\Phi(J)=\Z \cdot \frac{f}{2}T$, 
and $W$ is generated by $\m(F)$ and $\sqrt{q}=p^{f/2}$.  
\end{ppn}
\begin{proof}
Since $\s_{-1}=(\s_p)^{f/2} \in D$, we have $\Phi(\a)=\frac{wf}{2} T$ 
for any $\a\in W$ of weight $w$ by \eqref{eq:8-1}. 
Since we assumed $d \ge 3$, there exists a Jacobi sum of weight $1$ (e.g. $j_d(1,1)$), and the proposition follows. 
\end{proof}

For any $a\in\Z/d\Z$, define the Stickelberger element by 
\[\theta_d(a)= \sum_{h\in(\Z/d\Z)^*} \left\{-\frac{ha}{d}\right\} \s_h^{-1} \ \in\Q[G],\]
where $\{x\}$ denotes the fractional part, i.e. $x=\{x\}+\lfloor x \rfloor$. 
Note that 
\[\s_h \theta_d(a)=\theta_d(ha).\]
Define the trace element as $\wt T= \sum_{\s \in G} \s \in \Z[G]$. 
Then 
\begin{equation}\label{e8.1}
\theta_d(a)+\theta_d(-a) = \wt T \quad \text{if $a\ne 0$}.
\end{equation}
If $\pi\colon \Z[G] \to \Z[G/D]$ denotes the natural surjection, 
then $\pi(\wt T)= fT$. 
For $\ba=(a_1,\dots, a_n) \in A_d^{(n)}$, put 
\begin{align*}
\theta_d(\ba) 
&= \theta_d(a_1)+\cdots + \theta_d(a_n) - \theta_d(a_1+\cdots+a_n)
\\& =\theta_d(a_0)+\theta_d(a_1)+\cdots+\theta_d(a_n)-\wt T. 
\end{align*}
Then $\theta_d(\ba) \in \Z[G]$, and we have by \eqref{e8.1}
\begin{equation}
\theta_d(\ba)+\theta_d(-\ba) =(n-1)\wt T. 
\end{equation}

Stickelberger's theorem states that (see, e.g. \cite[Theorem 11.2.3]{berndtetal})
\begin{equation}\label{stickelberger}
(j_d(\ba)) = \theta_d(\ba) v \quad (\ba \in A_d^{(n)}).
\end{equation}
In other word, $\Phi(j_d(\ba))=\pi(\theta_d(\ba))$.  
Therefore, $\Phi(J)$ is generated as a $\Z$-module by 
$\{\theta_d(\ba) \mid  \ba \in A_d^{(2)} \}$.

Let $S \subset \Z[G]$ be the $\Z$-module generated by 
$\{\theta_d(\ba) \mid \ba \in A_d^{(2)} \}$.
Since $\s_h \theta_d(\ba)=\theta_d(h\ba)$, $S$ is an ideal of $\Z[G]$ 
and is called the \emph{Stickelberger ideal}. 
By \eqref{stickelberger}, we have $\pi(S)=\Phi(J)$.  
For any $G$-module $M$, write $M^\pm=\{m \in M \mid \s_{-1}m=\pm m\}$. 
By Iwasawa and Sinnott \cite{sinnott}, $S^-$ is of finite index in $\Z[G]^-$. 
More precisely, 
let $r$ be the number of prime factors of $d$, put $s=\max\{0, r-2\}$, 
and let $h_d^-$ be the minus part of the class number of $F=\Q(\z_d)$. 
Then 
\begin{equation}\label{iwasawa-sinnott}
m_d:=(\Z[G]^- : S^-) = 2^r h_d^-. 
\end{equation}

As a corollary, we obtain the following. 
\begin{ppn}\label{p8.2}
For any $\a \in W$, there exists $\z\in\m(F)$ such that $\z\a^{2m_d} \in J$. 
Moreover,  $J$ is of finite index in $W$  and we have
\begin{equation}
\label{eq:rank}
\rank J =
\rank W =
\begin{cases}
1+\frac{1}{2f}\vp(d) & \text{if $f$ is odd}, \\
1  & \text{if $f$ is even}.
\end{cases}
\end{equation}
\end{ppn}
\begin{proof}
It is a consequence of Proposition \ref{f even} if $f$ is even. Suppose that $f$ is odd. 
Then we have $\sigma_{-1} \not\in D$ and hence
the right vertical map in the commutative diagram
\[\xymatrix{S^- \ar@{^(_->}[r] \ar[d] & \Z[G]^- \ar@{->>}[d]\\
\Phi(J)^- \ar@{^(_->}[r] & \Z[G/D]^- 
}\]
is surjective.
By \eqref{iwasawa-sinnott}, we have $(\Z[G/D]^-:\Phi(J)^-) \mid m_d$. 
On the other hand, we have $\Phi(J)^+=\Phi(W)^+=\Z \cdot fT$.
For any $\a\in W$ of weight $w$,  we have by \eqref{eq:8-1}
\[
2\Phi(\a)=(\Phi(\a)+\s_{-1}\Phi(\a))+(\Phi(\a)-\s_{-1}\Phi(\a))
=wf T + (\Phi(\a)-\s_{-1}\Phi(\a)). 
\]
Since $fT=\Phi(q) \in \Phi(J)^+$ and $m_d(\Phi(\a)-\s_{-1}\Phi(\a)) \in \Phi(J)^-$, 
the assertion follows by \eqref{e8.0}
and  $\rank \Z[G/D]^-=\varphi(d)/(2f)$.
\end{proof}

We have an obvious inequality $(W:J) \le |\mu(F)| \cdot 2^{\rank W} \cdot m_d$.
It might be an interesting problem to find a better upper bound for $(W:J)$.
To understand its motivic meaning, 
we consider the group homomorphism
\[\sFr \colon \Pic(\Chow(\k,F)) \to W\]
which associates to an invertible motive $M$ its Frobenius eigenvalue in $\L\simeq \End(M)$. 
As a result of Theorem \ref{thm2} and Proposition \ref{p8.2}, we obtain the following. 
\begin{cor} 
For any $\a \in W$, there exist Fermat motives $M_i=h(F_d^{(2)})^{\bchi_i}$ 
($i=1,\dots, s$, $\bchi_i \in \mathfrak{X}_d^{(2)}$), 
an Artin motive $h(\Spec K)^\chi$ for a character $\chi \colon \Gal(K/\k) \to \m(F)$ of a finite extension $K/\k$,
and an integer $r$, such that 
\[\a^{2m_d} = \sFr \left(\bigotimes_{i=1}^s M_i \ot h(\Spec K)^\chi(r)\right).\]
\end{cor}

If we assume the conjectures of Beilinson and Tate as in the introduction, then $\sFr$ should be injective. 
It follows that any multiplicative relation among Weil numbers in the image of $\sFr$ should lift to a relation of 
invertible motives. In this paper, we have exhibited such lifts for basic relations in $J$ unconditionally. 

\begin{cor}\label{cor:conditional}
Assume that $\sFr$ is injective. 
Then 
$\Pic(\Chow(\k,F))_\Q$ is  isomorphic to $W_\Q$ and is generated 
as a $\Q$-vector space
by the classes of Fermat motives $h(F_d^{(2)})^{\bchi}$ ($\bchi \in \mathfrak{X}_d^{(2)}$). 
\end{cor}

Consider the normalization ($a \in \Z/d\Z$, $a \ne 0$)
\[\wt\theta_d(a)= \theta_d(a)-\frac{1}{2} \wt T \ \in \Q[G]. \] 
First, we have by \eqref{e8.1} 
\begin{equation}\label{e8.3}
\wt\theta_d(-a)=-\wt\theta_d(a), 
\end{equation}
i.e. $\wt\theta_d(a) \in \Q[G]^-$. 
Since $\Phi(J)^-_\Q=\Q[G]^-$ by \eqref{iwasawa-sinnott}, $\Q[G]^-$ is generated by $\{\wt\theta_d(a)\mid a\in\Z/d\Z, a \ne 0\}$. 
Secondly, for any positive divisor $n$ of $d$ and $a \in \Z/d\Z$ such that $na \ne 0$, we have  
\begin{equation}\label{e8.4}
\wt\theta_d(na)=\sum_{i=0}^{n-1} \wt\theta_d\left(a+\frac{d}{n}i\right). 
\end{equation}
(This is easy to prove and is also a consequence of \eqref{mult-j} and \eqref{stickelberger}.) 

\begin{ppn}\label{prop:basis}
The set
\[\left\{\wt\theta_d(a) \Bigm| a \in (\Z/d\Z)^*, \left\{\frac{a}{d}\right\}<\frac{1}{2}\right\}\] 
is a $\Q$-basis of $\Q[G]^-$. 
\end{ppn}

\begin{proof}
We already proved that $\Q[G]^-$ is generated by $\wt\theta_d(a)$
if $a$ ranges over $\Z/d\Z \setminus \{ 0 \}$.
One shows by induction using \eqref{e8.3} and \eqref{e8.4} that the set in the proposition already generates $\Q[G]^-$. 
It is a basis since  $\dim \Q[G]^-=\vp(d)/2$. 
\end{proof}

The last proposition shows that  the reflection formula \eqref{e8.3} and the multiplication formula \eqref{e8.4} generate all the $\Q$-linear relations among $\theta_d(\ba) \ (\ba \in A_d^{(n)}, n \ge 2 )$.
Hence the corresponding relations \eqref{e1-reflextion} and \eqref{mult-j} together with \eqref{eq:j-induction} generate all the relations among 
$j_d(\ba)$ up to torsion. 
Assuming the conjectures of Beilinson and Tate, it follows that 
\eqref{j-reflextion-mot} and Theorem \ref{thm1} (ii)  (under Proposition \ref{f-mot} (ii), (iii)) should generate all the relations among the Fermat motives up to torsion (i.e. up to powers and Artin motives). 

On the other hand, finding all the \emph{integral} relations appears to be subtler.
Indeed, it was first observed by Yamamoto \cite{Yamamoto} that
the relations among the Gauss sums are \emph{not} exhausted by
the reflection formula \eqref{gauss-reflection} and the multiplication formula \eqref{eq:mult-g},
disproving Hasse's conjecture.
Its counterpart for Jacobi sums can be found in \cite{MZ}.
We leave it as a future problem to study the corresponding relations among motives.

%%%%%
%\section*{Acknowledgements}

%%%%%%%%%%%%%%%%%

\end{document}